\documentclass[12pt]{amsart}

\usepackage{amsfonts,amssymb,stmaryrd,amscd,amsmath,latexsym,amsbsy,amsthm}

\usepackage{hyperref}

\usepackage{url}

\numberwithin{equation}{section}

\newtheorem{theorem}{Theorem}[section]
\newtheorem{lemma}[theorem]{Lemma}
\newtheorem{proposition}[theorem]{Proposition}
\newtheorem{corollary}[theorem]{Corollary}
\newtheorem{conjecture}[theorem]{Conjecture}
\theoremstyle{definition} 
\newtheorem{definition}[theorem]{Definition}

\theoremstyle{remark}
\newtheorem{remark}{Remark}[section]

\newcommand{\Aut}{\on{Aut}}

\newcommand{\g}{\mathfrak{g}}

\newcommand{\h}{\mathfrak{h}}

\newcommand{\ov}{\overline}

\newcommand{\la}{\lambda}

\newcommand{\bs}{\backslash}
\newcommand{\bl}{\lambda}

\newcommand{\C}{{\mathbb C}}

\newcommand{\ben}{\begin{enumerate}}
\newcommand{\een}{\end{enumerate}}

\newcommand{\nc}{\newcommand}
\nc{\on}{\operatorname}
\nc{\wh}{\widehat}
\nc{\wt}{\widetilde}
\nc{\sw}{{\mathfrak s}{\mathfrak l}}
\nc{\ghat}{\wh{\g}}
\nc{\hhat}{\wh{\h}}
\nc{\mc}{\mathcal}
\nc{\Bun}{\on{Bun}}
\nc{\ol}{\overline}
\nc{\OO}{\mathcal O}
\nc{\pone}{{\mathbb P}^1}
\nc{\pa}{\partial}
\nc{\Pic}{\on{Pic}}
\nc{\ga}{\gamma}
\nc{\orr}{\underline}
\nc{\mbb}{\mathbb}
\nc{\mbf}{\mathbf}
\nc{\V}{{\mc V}}
\nc{\su}{\wh\sw_2}

\newcommand{\norm}[1]{\left\lVert#1\right\rVert}

\theoremstyle{plain}

\newtheorem*{sol}{Solution}

\usepackage{color}

\newcommand{\solu}[1]{\begin{sol}{\bf (\ref{#1})}}
\newcommand{\mP}{\mathbb P}

\newcommand{\mcA}{\mathcal A}
\newcommand{\mcV}{\mathcal V}
\newcommand{\mcB}{\mathcal B}

\newcommand{\mcP}{\mathcal P}

\newcommand{\mcH}{\mathcal H}
\newcommand{\mcS}{\mathcal S}

\newcommand{\mcL}{\mathcal L}

\newcommand{\mcM}{\mathcal M}
\newcommand{\mcF}{\mathcal F}

\nc{\Bunc}{\on{Bun}^{\on{s}}}
\nc{\Bunn}{\on{Bun}^{\circ}}

\def\g{\mathfrak{g}}
\def\R{\mathbb{R}}

\def\D{\mathcal{D}}
\def\bO{\Omega}

\def\h{\mathfrak{h}}

\def\Z{\mathbb{Z}}

\def\M{\mathcal{M}}

\def\Gr{\on{Gr}}

\def\Fq{\mathbb{F}_q}
\def\LG{{}^L\hspace*{-0.4mm}G}
\def\lg{{}^L\hspace*{-0.4mm}\g}

\nc{\ppart}{(\!(t)\!)}
\nc{\zpart}{(\!(z)\!)}
\nc{\hl}{h_{\ell}}
\nc{\hr}{h_{r}}
\nc{\mb}{\mathbf}
\nc{\Op}{\on{Op}}
\nc{\al}{\alpha}
\nc{\Derp}{\on{Der}^+\C[[z]]}
\nc{\rel}{\on{rel}}
\nc{\om}{\omega}
\nc{\dt}{{\bullet}}
\nc{\ud}{^{-1}}
\nc{\ld}{_*}
\nc{\lr}{\leftarrow}
\nc{\ds}{\displaystyle}
\nc{\mh}{{\mathbb H}}
\nc{\ep}{\epsilon}

\newcommand{\YY}{\mathbb{Y}}

\newcommand{\mF}{\mathbb F}
\usepackage{xcolor}

\begin{document}

\title[Analytic Langlands correspondence for curves over local
  fields]{Hecke operators and analytic Langlands correspondence for
  curves over local fields}

\author{Pavel Etingof}

\address{Department of Mathematics, MIT, Cambridge, MA 02139, USA}

\author{Edward Frenkel}

\address{Department of Mathematics, University of California,
  Berkeley, CA 94720, USA}

\author{David Kazhdan}

\address{Einstein Institute of Mathematics, Edmond J. Safra Campus,
  Givaat Ram The Hebrew University of Jerusalem, Jerusalem, 91904,
  Israel}

\dedicatory{In memory of Isadore Singer}

\begin{abstract}
We construct analogues of the Hecke operators for the moduli space of
$G$-bundles on a curve $X$ over a local field $F$ with parabolic
structures at finitely many points. We conjecture that they define
commuting compact normal operators on the Hilbert space of
half-densities on this moduli space. In the case $F=\C$, we also
conjecture that their joint spectrum is in a natural bijection with
the set of $\LG$-opers on $X$ with real monodromy. This may be viewed
as an analytic version of the Langlands correspondence for complex
curves. Furthermore, we conjecture an explicit formula relating the
eigenvalues of the Hecke operators and the global differential
operators studied in our previous paper \cite{EFK}. Assuming the
compactness conjecture, this formula follows from a certain system of
differential equations satisfied by the Hecke operators, which we
prove in this paper for $G=PGL_n$.
\end{abstract}

\maketitle

\setcounter{tocdepth}{2}
\tableofcontents

\section{Introduction}

The Langlands Program for a curve $X$ over a finite field $\mF_q$
studies, in the unramified case, the joint spectrum of the commuting
Hecke operators acting on the space of $L^2$ functions on the groupoid
of $\mF_q$-points of the stack $\Bun_G$ of $G$-bundles on $X$ with its
natural measure. It aims to express this spectrum in terms of the
Galois data associated to $X$ and the Langlands dual group $\LG$.

In this paper, which is a continuation of \cite{EFK}, we study an
analogue of this problem when $\mF_q$ is replaced by a local field
$F$. We define analogues of Hecke operators on a dense subspace of
the Hilbert space $\mcH_G$ of half-densities on $\Bun_G$ and
conjecture that they extend to commuting compact normal operators
on $\mcH_G$. Investigation of these operators is a hybrid of
functional analysis and the study of the algebraic structure of
$\Bun_G$.

In the case $F=\C$, these Hecke operators commute with the global
holomorphic differential operators on $\Bun_G$ introduced in
\cite{BD}, as well as their complex conjugates. This enables us to
describe, subject to the Compactness Conjecture formulated below, the
joint spectrum of these operators in terms of the $\LG$-{\em opers} on
$X$ with {\em real monodromy}, which play the role of the Galois data
appearing in the case of curves over a finite field. We consider this
description as an analogue of the Langlands correspondence in the case
of complex curves. In the case of $G=SL_2$, a similar conjecture about
eigenfunctions of the differential operators on $\Bun_G$ was proposed
earlier by J. Teschner \cite{Teschner}.

\subsection{Main objects}    \label{mainobj}

Let $F$ be a local field, $G$ a split connected reductive group over
$F$, $B$ its Borel subgroup, $X$ a smooth projective curve over $F$
and $S\subset X$ a reduced divisor defined over $F$. We denote by
$\Bun_G(X,S)$ the moduli stack of pairs $(\mcF ,r_S)$, where $\mcF$ is
a $G$-bundle on $X$ and $r_S$ is a reduction to $B$ of the restriction
$\mcF|_{S}$ of $\mcF$ to $S$.\footnote{Just as in the case of curves
over $\Fq$, the theory naturally extends to incorporate reductions to
other parabolic subgroups of $G$. We also expect that it extends to
wild ramification.}

Let $\Bunc_G(X,S)\subset \Bun_G(X,S)$ be the substack of {\em
  regularly stable} pairs $(\mcF ,r_S)$, i.e. stable pairs whose group
of automorphisms is equal to the center $Z(G)$ of $G$.

We will assume throughout this paper that $\Bunc_G(X,S)$ is {\em open
  and dense} in $\Bun_G(X,S)$, which means that we are considering one
of the following cases:

\begin{enumerate}

\item the genus of $X$ is greater than 1, and $S$ is arbitrary;

\item $X$ is an elliptic curve and $|S| \geq 1$;

\item $X=\mP^1$ and $|S| \geq 3$.

\end{enumerate}

The stack $\Bunc_G(X,S)$ is a $Z(G)$-gerbe over a smooth analytic
$F$-variety $\Bunn_G(X,S)$, which is the corresponding coarse moduli
space. For our purposes, $\Bunn_G(X,S)$ is a good substitute for
$\Bunc_G(X,S)$ because all objects we need (such as line bundles or
differential operators) naturally descend from $\Bunc_G(X,S)$ to
$\Bunn_G(X,S)$.

R.P. Langlands asked in \cite{L:analyt} whether it is possible to
develop an analytic theory of automorphic functions for complex
curves. In \cite{EFK} and the present paper we are developing such a
theory (for curves over $\C$ and other local fields), understood as a
spectral theory of commuting Hecke operators acting on the Hilbert
space $\mcH_G(X,S)$ of half-densities on $\Bunn_G(X,S)$ for pairs
$(X,S)$ satisfying the above condition, as well as global differential
operators on $\Bun_G(X,S)$ when $F=\C$ or $\R$.

To define the Hilbert space $\mcH_G(X,S)$, we introduce the
$\R_{>0}$-line bundle $\Omega^{1/2}_{\Bun}$ of half-densities on
$\Bunn_G(X,S)$.

Namely, our local field $F$ is equipped with the norm map $x\mapsto
\norm{x}$ (the Haar measure on $F$ multiplies by $\norm{\lambda}$
under rescaling by $\lambda\in F$). For instance, for $F={\mathbb R}$
we have $\norm{x}=|x|$, for $F=\C$ we have $\norm{x}=|x|^2$, and for
$F=\Bbb Q_p$ we have $\norm{x}=p^{-v(x)}$, where $v(x)$ is the
$p$-adic valuation of $x$ (see \cite{We}).  Using the norm map, we
associate to any line bundle $\mcL$ on any smooth algebraic
$F$-variety $\YY$ a complex line bundle $\norm{\mcL}$ on the analytic
$F$-variety $Y:=\YY (F)$ with the structure group $\R_{>0}$. Namely,
if the transition functions of $\mcL$ are $\{ g_{\alpha\beta} \}$,
then the transition functions of $\norm{\mcL}$ are $\{
\norm{g_{\alpha\beta}} \}$ (in particular, if $F=\C$, then these are
$\{ |g_{\alpha\beta}|^2 \}$). (Note that for archimedian fields the
line bundle $\norm{\mcL}$ has a $C^\infty$-structure.)

It is shown in \cite{BD}, Sect. 4 (see also \cite{LS}) that the
canonical line bundle $K_{\Bun}$ over $\Bun_G(X,S)$ has a square root
$K_{\Bun}^{1/2}$. The restriction of $K_{\Bun}^{1/2}$ to
$\Bunc_G(X,S)$ descends to a line bundle on $\Bunn_G(X,S)$ for which
we will use the same notation $K_{\Bun}^{1/2}$. We then set
\begin{equation}    \label{om12}
\Omega^{1/2}_{\Bun} := \norm{K_{\Bun}^{1/2}}.
\end{equation}
Alternatively, $\Omega^{1/2}_X$ can be defined as the square root of
the line bundle $\norm{K_X}$ (since the structure group of the latter
is $\R_{>0}$, this is well-defined). This shows that line bundle
$\Omega^{1/2}_{\Bun}$ does not depend on the choice of $K_{\Bun}^{1/2}$.

Let $V_G(X,S)$ be the space of smooth compactly supported sections of
$\Omega ^{1/2}_{\Bun}$ over $\Bunn_G(X,S)(F)$. We denote by $\langle
\cdot,\cdot \rangle$ the positive-definite Hermitian form on
$V_G(X,S)$ given by
$$
\langle v,w \rangle :=\int _{\Bunn_G(X,S)(F)} v \cdot \overline{w},
\qquad v,w\in V_G(X,S).
$$
and define $\mcH_G(X,S)$ as the {\em Hilbert space} completion of
$V_G(G,S)$.

\subsection{Hecke operators}    \label{notation}

{}From now on, for brevity, we will drop $(X,S)$ in our notation when
no confusion could arise (i.e. write $\Bun_G$ for $\Bun_G(X,S)$,
$\mcH_G$ for $\mcH_G(X,S)$, etc.). We are going to define analogues of
the Hecke operators for curves over a local field.

For non-archimedian local fields, these operators were suggested
by A. Braverman and one of the authors in \cite{BK}. For
$G=PGL_2, X=\pone$, Hecke operators were studied by M. Kontsevich in
\cite{Ko}, and he also knew that such operators could be defined in
general. In his letters to us (2019) he conjectured compactness and
the Hilbert--Schmidt property of averages of the Hecke operators over
sufficiently many points.

The idea that Hecke operators over $\C$ could be used to construct an
analogue of the Langlands correspondence was suggested by
R.P. Langlands \cite{L:analyt}, who sought to construct them in the
case when $G=GL_2$, $X$ is an elliptic curve, and $S = \emptyset$
(however, for an elliptic curve $X$ we can only define Hecke operators
if $S \neq \emptyset$; see \cite{F:analyt}, Sect. 3).

In the case when $G$ is a torus, the Hecke operators and their spectra
on $\mcH_G$ were completely described in \cite{F:analyt}, Sect. 2. We
recall these results for $G=GL_1$ in Section \ref{globargGL1}
below.

From now on (except in Section \ref{GL1} below) we will assume that
$G$ is a connected simple algebraic group over $F$. All of our results
and conjectures generalize in a straightforward way to connected
semisimple algebraic groups over $F$.

Let $P^\vee_+$ be the set of dominant integral coweights of $G$. For
$\la \in P^\vee_+$ we denote by $\ol{Z}(\lambda)$ the Hecke
correspondence
$$
\ol{q}: \ol{Z}(\lambda)\to \Bun_G \times \Bun_G \times (X \bs S)
$$
describing the $\la$-modifications of stable $G$-bundles at points of
$X \bs S$. More precisely, $\ol{Z}(\lambda)$ classifies triples
$(\mcF,\mcP,x,t)$, where $\mcF$ and $\mcP$ are principal $G$-bundles on
$X$ with a reduction to $B$ at $S$, $x$ is a closed point in $X \bs
S$, and
$$
t: \mcF|_{X \bs x} \overset\sim\longrightarrow \mcP|_{X \bs x}
$$
is an isomorphism satisfying the following condition. Choose a
formal coordinate $z$ at $x$ and a trivialization of $\mcF$
on the formal disc $D_x := \on{Spec} \C[[z]]$; then the restriction of
$t$ to $D_x^\times := \on{Spec} \C\zpart$ naturally gives rise to a
point in the affine Grassmannian $\Gr = G\zpart/G[[z]]$. The condition
is that this point belongs to the closure $\ol\Gr_\la$ of the
$G[[z]]$-orbit $\Gr_\la = G[[z]] \cdot \la(z)G[[z]]$ (this condition
does not depend on the above choices). Further, denote by $Z(\la)$
the open part of $\ol{Z}(\la)$ satisfying the condition that this
point belongs to $\Gr_\la$ itself, and let $q$ be the restriction of
$\ol{q}$ to $Z(\la)$.

Let
$$
p_{1,2}: \Bun_G \times \Bun_G \times (X\bs S) \to \Bun_G
$$
$$
p_3: \Bun_G \times \Bun_G \times (X \bs S) \to X \bs S
$$
be the natural projections, $\ol{q}_i:=p_i\circ \ol{q}$, and
$q_i:=p_i\circ q$. Thus, the fibers of the morphism $\ol{q}_2 \times
\ol{q}_3$ are isomorphic to $\ol\Gr_\la$ and the fibers of the
morphism $q_2 \times q_3$ are isomorphic to $\Gr_\la$. Let $K_2$ be
the relative canonical line bundle of the morphism $q_2\times q_3$.

Let $\rho$ be the integral weight of $G$ such that $\langle
\al_i^\vee,\rho \rangle = 1$ for all simple coroots $\al_i^\vee$ of
$G$. If the set $\{ \langle \la,\rho \rangle, \la \in P^\vee_+ \}$
only contains integers, then Beilinson and Drinfeld construct in
\cite{BD}, Sect. 4 a square root $K_{\Bun}^{1/2}$ of $K_{\Bun}$
equipped with a trivialization of its fiber at the trivial
$G$-bundle. If this set contains half-integers, they construct such
an object for each choice of a square root $K^{1/2}_X$ of $K_X$. In
the latter case, we will assume throughout that a choice of
$K^{1/2}_X$ has been made. We will denote by $\ga$ the isomorphism
class of this $K^{1/2}_X$ and by $K_{\Bun}^{1/2}$ the corresponding
line bundle on $\Bun_G$ (equipped with a trivialization of its fiber
at the trivial $G$-bundle). The following result is essentially due to
Beilinson and Drinfeld \cite{BD}. In Section \ref{proofh} below we
explain how to derive it from formula (241) of \cite{BD} (see also
\cite{BK}, Theorem 2.4).

\begin{theorem}\label{h}
There exists a canonical isomorphism
\begin{equation}    \label{cana}
a : q_1^* (K_{\rm Bun}^{1/2}) \overset{\sim}\longrightarrow q_2^*
(K_{\rm Bun}^{1/2})\otimes K_2 \otimes
q_3^* (K_X ^{-\langle\bl ,\rho\rangle})
\end{equation}
where $\rho$ is the half-sum of positive roots. 
\end{theorem}

The isomorphism $a$ gives rise to an isomorphism
\begin{equation}    \label{cana2}
a^2 : q_1^* (K_{\rm Bun}) \overset{\sim}\longrightarrow q_2^*
(K_{\rm Bun})\otimes (K_2)^2 \otimes
q_3^* (K_X ^{-2\langle\bl ,\rho\rangle})
\end{equation}
which does not depend on the choice of $K^{1/2}_X$ (and
$K^{1/2}_{\Bun}$). Using the formula $\norm{a}=\norm{a^2}^{1/2}$, we
then obtain a canonical isomorphism
\begin{equation}    \label{moda}
\norm{a} : q_1^* (\Omega_{\rm Bun}^{1/2}) \overset{\sim}\longrightarrow
q_2^* (\Omega_{\rm Bun}^{1/2})\otimes \bO_2 \otimes
q_3^* (\Omega_X^{-\langle\bl ,\rho\rangle}).
\end{equation}
Here $\Omega^{1/2}_X$ (resp. $\Omega_{\rm Bun}^{1/2}$) is defined as
the square root of the line bundle $\norm{K_X}$ (resp.
$\norm{K_{\Bun}}$); since the structure group of the latter is
$\R_{>0}$, these square roots are well-defined. Also, $\bO_2 :=
\norm{K_2}$, the $C^\infty$ line bundle of densities along the
fibers of $q_2\times q_3$.

Thus, $\norm{a}$ does not depend on the choice of $K^{1/2}_X$ (and
$K_{\rm Bun}^{1/2}$). This implies that our definition of the Hecke
operators given below also does not depend on this choice. However,
we need to make these choices to describe the algebra of global
differential operators on $\Bun_G$ (see Section \ref{caseC}) and to
state the differential equations \eqref{eqHom1} satisfied by the Hecke
operators. Hence, for the sake of uniformity of exposition, we have
made these choices from the beginning.

Now let
\begin{equation}    \label{UG}
U_G(\la) := \{ \mcF \in \Bunn_G \, | \, q_2(q_1^{-1}(\mcF)) \subset
\Bunn_G \}.
\end{equation}
This is an open subset of $\Bunn_G$, which is dense if the dimension
of $\Bun_G$ is sufficiently large (for instance, for $G=PGL_2$ it is
sufficient that $\on{dim} \Bun_{PGL_2} > 1$).

Suppose that $U_G(\la) \subset \Bunn_G$ is dense. Let $V_G(\la)\subset
V_G$ be the subspace of half-densities $f$ with compact support such
that $\on{supp}(f)\subset U_G(\la)$. For $\mcP \in \Bunn_G(F)$, let
$$
Z_{\mcP,x} := (q_2\times q_3)^{-1}(\mcP \times x),
\qquad x \in (X \bs S)(F).
$$
According to the above definition of $Z(\la)$, the variety
$Z_{\mcP,x}$ is isomorphic to the
$G[[z]]$-orbit $\Gr_\la$ in the affine Grassmannian of $G$. Hence it
is proper and has rational singularities. The isomorphism \eqref{moda}
and the results of Theorem 2.5 and Sect. 2.6 of \cite{BK} imply that
for any $f\in
V_G(\la)$ and $x\in (X \bs S)(F)$, the restriction of the pull-back
$q_1^*(f)$ to $Z_{\mcP,x}$ is a compactly supported measure with
values in the line $(\Omega_{\rm Bun}^{1/2})_\mcP \otimes
(\Omega_X^{-\langle\bl ,\rho\rangle})_x$. Therefore the integral
\begin{equation}    \label{intH}
(\wh{H}_\la(x) \cdot f)(\mcP) := \int_{Z_{\mcP,x}} q_1^*(f),
\end{equation}
is absolutely convergent for all $f\in V_G(\la)$ and $\mcP\in
\Bunn_G(F)$. In fact, it belongs to $V_G \otimes (\Omega_X^{-\langle\bl
  ,\rho\rangle})_x$.

Since $V_G \subset \mcH_G$, this integral defines a {\em Hecke
  operator}
$$
\wh{H}_\la(x): V_G(\la) \to \mcH_G \otimes (\Omega_X^{-\langle\bl
  ,\rho\rangle})_x.
$$
As $x$ varies along $X \bs S$, the operators $\wh{H}_\la(x)$ combine
into a section of the line bundle $\Omega_X^{-\langle\bl ,\rho\rangle}$ over $X
\bs S$ with values in operators $V_G(\la) \to \mcH_G$. We denote it by
$\wh{H}_\la$.

\begin{remark}
We conjecture (and can prove in a number of cases) that the integrals
defining $\wh{H}_\la (f)$ are absolutely convergent for all $f\in
V_G$.\qed
\end{remark}

  \begin{remark}
  There are two basic differences between the cases of finite and
  local fields.  In the case of a curve $X$ over a finite field $F$,
  the set $\Bun_G(F)$ of isomorphism classes $[\mcP]$ of $G$-bundles
  is a countable set with a natural (Tamagawa) measure such that the
  measure of $[\mcP]$ is equal to $1/|\on{Aut}(\mcP)(F)|$. On the
  other hand, for a curve $X$ over a local field $F$, neither
  $\Bun_G(F)$ nor the fibers of the Hecke
  correspondence carry any natural measures. To overcome this problem, we
  replace the space of functions on $\Bun_G(F)$ by the space of
  half-densities and use the isomorphism \eqref{moda} to write down
  formula \eqref{intH} for the action of Hecke operators.

The second difference is that for a curve over a finite field $F$ the
integrals \eqref{intH} are finite sums and so the corresponding Hecke
operators are well-defined on the space of all functions on
$\Bun_G(F)$. On the other hand, for a curve $X$ over a local field $F$
there is no obvious non-trivial space of half-densities stable under
these Hecke operators. One could try to consider the space of
continuous half-densities with respect to the natural topology on
$\Bun_G(F)$ but this space is equal to $\{0\}$. The reason is that the
natural topology on $\Bun_G(F)$ is not Hausdorff.\footnote{ For a
simply-connected group $G$, the closures of any two points in
$\Bun_G(F)$ have a non-empty intersection, and therefore the only
continuous functions on $\Bun_G(F)$ are constants.}

The absence of such a subspace creates serious analytic difficulties
in defining the notion of the spectrum of the Hecke operators because
they are initially defined only on a dense subspace $V_G(\la)$ of
$\mcH_G$ (under our assumption that $\Bunc_G(X,S)$ is dense in
$\Bun_G(X,S)$). However, we expect (and can prove for $G=PGL_2,
X=\pone$ \cite{EFKnew}) that these operators extend to bounded
operators on $\mcH_G$, which are moreover compact and normal, as we
state in the following conjecture.\qed
\end{remark}

\begin{conjecture}\label{com}
Suppose that $\la \neq 0$.
\begin{enumerate}
\item For any identification $(\Omega_X^{1/2})_x \cong
  \C$, the operators $\wh{H}_\la(x): V_G(\la) \to \mcH_G$ extend to a
  family of commuting compact normal operators on $\mcH_G$, which we
  denote by $H_\lambda(x)$.
  \item $H_\lambda(x)^\dagger=H_{-w_0(\lambda)}(x)$.
  \item $\bigcap _{\bl,x}{\rm Ker}H_\bl (x)= \{ 0 \}$.
\end{enumerate}
\end{conjecture}

\begin{remark}
  (i) It is easy to see that
  $$
  \langle H_\la f,g \rangle = \langle f,H_{-w_0(\la)} g \rangle
  $$
for all $f \in V_G(\la), g\in V_G(-w_0(\lambda))$. Therefore part (2)
of the conjecture immediately follows from part (1).

(ii) If $S \neq \emptyset$, we expect that the statement of part (3)
with a fixed $\la$ can be obtained from (1) by considering the limit
of $H_\la(x)$ when $x$ tends to a point of $S$. This will be discussed
in more detail in a follow-up paper.\qed
\end{remark}

We will refer to Conjecture \ref{com} as the Compactness
Conjecture. In our next paper \cite{EFKnew}, we will prove this
conjecture in the case $G=PGL_2$, $X=\pone$, $|S|>3$. {\em From now
  on, we will assume the validity of the Compactness Conjecture.}

Denote by $\mh_G$ the commutative algebra generated by operators
$H_\la (x), \la \in P^\vee_+, x\in (X \bs S)(F)$, and by
$\on{Spec}(\mh_G)$ their joint spectrum.

\begin{corollary}\label{directsum}
  We have an orthogonal decomposition
$$
\mcH_G = \underset{s \in \on{Spec}(\mh_G)}{\widehat\bigoplus}
\mcH_G(s),
$$
where $\mcH_G(s), s \in \on{Spec}(\mh_G)$, are
finite-dimensional joint eigenspaces of $\mh_G$ in $\mcH_G$.
\end{corollary}

Let $\mh_G(x)$ be the subalgebra of $\mh_G$ generated by $H_\la (x),
\la \in P^\vee_+$, for a fixed $x \in (X\bs S)(F)$.

\begin{proposition}[\cite{BK}]    \label{algiso}
  There is an algebra isomorphism $\mh_G(x) \simeq
  \C[P^\vee_+]$. In particular, $H_\la(x) \cdot H_\mu(x) =
  H_{\la+\mu}(x)$.
\end{proposition}

\begin{proof}
If $F$ is a non-archimedian field, this is equation (3.4) of
\cite{BK}, which is proved in Lemmas 3.5 and 3.9 of \cite{BK}. The
same proof works for a general local field.
\end{proof}

\begin{remark}
Note the difference with the case of a curve over $\Fq$, where the
Satake isomorphism naturally identifies the Hecke algebra at a point
$x$ with $\on{Rep} \LG$. Thus, in this case the product of the Hecke
operators corresponding to $\la, \mu \in P^\vee_+$ is in general
{\em not} equal to the Hecke operator corresponding to $\la+\mu$;
there are correction terms corresponding to lower weights.\qed
\end{remark}

\begin{remark}
Using an analogous correspondence
$$
Z_r(\lambda)\to \Bun_G \times \Bun_G \times \on{Sym}^r(X \bs
S)
$$
for a positive integer $r>1$, the above construction can be
generalized to yield operators $H_\la(D), D \in \on{Sym}^r(X \bs
S)(F)$. We expect that an analogue of the Compactness Conjecture
\ref{com} holds for them. Moreover, if $D=\sum_i m_ix_i$, $x_i\in
X(F)$ then $H_\lambda(D)=\prod_i H_\lambda(x_i)^{m_i}$.\qed
\end{remark}

\subsection{The case $F=\C$}    \label{caseC}

At the moment, we have a conjectural description of the spectra
$\on{Spec}(\mh_G)$ only if $F=\C$ and, in some cases, for $F=\R$. We
will now describe our conjecture for $F=\C$. Note that in this case we
have $\norm{z} = |z|^2$. For simplicity, we only consider here the
case when $S=\emptyset$ (and hence $g>1$). But all of our results and
conjectures have natural generalizations to the case of an arbitrary
$S$. The case of $G=PGL_2$, $X=\pone$, and $|S|>3$ will be discussed
in detail in our next paper \cite{EFKnew}.

In our previous paper \cite{EFK} we studied the action of global
holomorphic and anti-holomorphic differential operators on $\Bun_G$ on
a dense subspace of the Hilbert space $\mcH_G$ in the case of a
simply-connected simple group $G$. Here we generalize the setup to the
case of an arbitrary connected simple $G$. Then $\Bun_G$ has finitely
many connected components labeled by $\pi_1(G)$, which we will denote
by $\Bun^\beta_G, \beta \in \pi_1(G)$.

\begin{definition}    \label{beta}
{\em Let $D^\beta_G$ be the algebra of global algebraic (hence
  holomorphic) differential operators acting on the line bundle
  $K_{\Bun}^{1/2}$ over a connected component $\Bun^\beta_G$ of
  $\Bun_G$. Beilinson and Drinfeld have proved in \cite{BD} that these
  algebras are isomorphic to each other. Therefore, we will also use
  the notation $D_G$ for $D^\beta_G$.}
\end{definition}

On the other hand, let ${\rm Op}_{\LG}(X)$ be the space of
(holomorphic) $\LG$-opers on $X$ defined in \cite{BD,BD:opers}. As
shown in \cite{BD}, Sect. 3.4, ${\rm Op}_{\LG}(X)$ is a union of
connected components, each isomorphic to the affine space ${\rm
  Op}_{\LG_{\on{ad}}}(X)$, where $\LG_{\on{ad}} = \LG/Z({}\LG)$ is the
group of adjoint type associated to $\LG$ (here $Z({}\LG)$ denotes the
finite center of $\LG$). The group $H^1(X,Z({}\LG))$ naturally acts on
the set of these components by changing the underlying $\LG$-bundle,
and this action is simply-transitive. As shown in \cite{BD},
Sect. 3.4.2, if the set $\{ \langle \la,\rho \rangle, \la \in P^\vee_+
\}$ only contains integers, then there is a canonical choice of a
component, and if the set $\{ \langle \la,\rho \rangle, \la \in
P^\vee_+ \}$ contains half-integers, then there is a canonical choice
of a component ${\rm Op}^{\ga}_{\LG}(X)$ for each isomorphism class
$\ga$ of $K^{1/2}_X$.

For the sake of uniformity of notation, in both cases we will denote
the canonical component by ${\rm Op}^\ga_{\LG}(X)$, with the
understanding that in the latter case $\ga$ is the isomorphism class
of $K^{1/2}_X$ we have chosen before Theorem \ref{h}. The following
theorem is proved in \cite{BD} (Theorem 3.3.2 and Sects. 2.2.5 and
2.7.4).

\begin{theorem}    \label{DGOp}
The algebra $D_G$ is commutative and ${\rm Spec}(D_{G})$ is isomorphic
to the affine space ${\rm Op}^\ga_{\LG}(X)$.
\end{theorem}

For a given $\chi \in {\rm Op}^\ga_{\LG}(X)$, the system
of differential equations for the eigenfunctions of this algebra,
\begin{equation}    \label{qH}
  P \cdot f = \chi(P) f, \qquad P \in D_G,
\end{equation}
where $f$ denotes a local holomorphic section of $K_{\Bun}^{1/2}$, is
known as the {\it quantum Hitchin integrable system.} Its local
solutions are the same as the homomorphisms from the twisted
$D$-module $\Delta_\chi$ on $\Bun_G$ to $K_{\Bun}^{1/2}$, where
\begin{equation}    \label{Deltachi}
  \Delta_\chi := \D'_{\Bun_G} \underset{D_G}\otimes \C_\chi
\end{equation}
and $\D'_{\Bun_G}$ is the sheaf of (holomorphic) differential operators
acting on $K_{\Bun}^{1/2}$ (in \eqref{Deltachi} we consider the
diagonal embedding $D_G \to \D'_{\Bun_G}$ taking the sum over all
components $\Bun^\beta_G$ of $\Bun_G$).

\begin{remark}
  If $S \neq \emptyset$, the algebra of global holomorphic
  differential operators acting on $K_{\Bun}^{1/2}$ is
  non-commutative, but it contains a commutative subalgebra whose
  spectrum is isomorphic to the space of $\LG$-opers on $X$ with
  regular singularity and regular unipotent monodromy at the points
  of $S$ (see \cite{EFK}, Sect. 6). There is a similar algebra in the
  case of higher level structures at $S$ (this is analogous to the
  wild ramification in the case of curves over $\Fq$), with the
  corresponding $\LG$-opers having irregular singularities at the
  points of $S$. We expect that the theory has a generalization to
  this case as well.\qed
\end{remark}

It is shown in \cite{BD} that $\Delta_\chi$ is a holonomic
$D$-module. It also has regular singularities (see Theorem
\ref{RS},(1) below). Moreover, there is an open substack
$\Bun^{\on{vs}}_G$ of {\em very stable bundles} (i.e. $G$-bundles
${\mc P}$ such that the vector bundle $\g_{\mcF} \otimes K_X$ does not
admit non-zero sections taking nilpotent values everywhere on $X$)
such that the restriction of $\Delta_\chi$ to $\Bun^{\on{vs}}_G$ is a
vector bundle with a projectively flat connection. But its rank grows
exponentially with the genus of the curve $X$ and it has highly
non-trivial monodromy around the divisor $\Bun_G \bs
\Bun^{\on{vs}}_G$. Therefore, it does not make sense to look for
individual holomorphic solutions of the system \eqref{Deltachi}.

However, as explained in \cite{Teschner} and in \cite{EFK}, Sect. 1.5,
it makes sense to couple the system \eqref{Deltachi} to its
anti-holomorphic analogue and look for single-valued solutions of the
resulting system of differential equations on the locus
$Bun^{\on{vs}}_G$ of very stable bundles in $\Bunn_G$. These are the
automorphic functions of the analytic theory. It is natural to try to
interpret them as eigenfunctions of the algebra
$$
\mcA_G := D_G \underset{\C}\otimes \ol{D}_G.
$$
This is a non-trivial task because elements of $\mcA_G$ correspond to
unbounded operators on the Hilbert space $\mcH_G$. Initially, they
are defined on the dense subspace $V_G \subset \mcH_G$. In \cite{EFK}
(Definition 1.7) we introduced a Schwartz space $S(\mcA_G) \subset
\mcH_G$ and conjectured that the elements of $\mcA_G$ can be extended
to $S(\mcA_G)$ so that their closures form a family of commuting
normal operators. Moreover, we conjectured that their joint spectrum
is the set $\on{Op}^\ga_{\LG}(X)_{\R}$ of $\LG$-opers on $X$ with real
monodromy (see below). In the case of $G=SL_2$, a similar conjecture
was proposed earlier in \cite{Teschner}.

In \cite{EFK} we were able to prove these conjectures in the simplest
non-trivial case. However, in the general case it is a daunting task
to prove them directly. This is where the integral Hecke operators
come in handy. Like differential operators, they are also initially
defined on a dense subspace of the Hilbert space $\mcH_G$. But unlike
the differential operators, we do expect Hecke operators to extend
to mutually commuting continuous operators on the entire $\mcH_G$,
which are moreover {\em normal, compact}, and have trivial common
kernel. This is the statement of our Compactness Conjecture \ref{com}.
It implies that $\mcH_G$ decomposes into a (completed) direct sum of
mutually orthogonal {\em finite-dimensional eigenspaces} of the Hecke
operators.

Next, we would like to say that the algebra $\mcA_G$ preserves the
subspaces $\mcH_G(s)$. To do this, observe that $\mcA_G$ acts on the
space $V_G^\vee$ of distributions on $\Bunn_G$. It follows from the
definition of $\mcH_G$ that $\mcH_G$ is naturally realized as a
subspace of $V_G^\vee$. Hence we can apply elements of $\mcA_G$ to
vectors in the eigenspaces ${\mbb H}_G(s)$ of the Hecke operators,
viewed as {\em distributions}.

\begin{conjecture}    \label{eigA}
  Every eigenspace ${\mbb H}_G(s)$ of the Hecke operators, viewed as a
  subspace of $V_G^\vee$, is an eigenspace of ${\mc A}_G$.
\end{conjecture}

\begin{remark} (i) There is a weaker version of Conjecture
  \ref{eigA}, in which $V_G$ is replaced by the space of smooth
  functions with compact support on some open dense set $U\subset
  \Bunn_G$. For practical purposes such a weak version is almost as
  good but might be easier to prove. For example, we may take $U$ to
  be the above locus $Bun^{\on{vs}}_G$ of very stable bundles in
  $\Bunn_G$, on which the eigenfunctions of the Hecke operators are
  smooth (in fact, this is true for any local field $F$; we will
  explain this in a follow-up paper). Then the conjecture is that
  these eigenfunctions satisfy the differential equations of the
  quantum Hitchin system combined with its complex conjugate system
  (which are smooth on this locus) in the classical sense.

(ii) The algebra ${\mc A}_G$ acts on both $V_G(\la)$ and $V_G$. Using
an identification \cite{BD} of $D_G$ with a quotient of the center at
the critical level \cite{FF,F:wak}, one can show that the action of
${\mc A}_G$ commutes with the Hecke operators $\wh{H}_\la(x): V_G(\la)
\to V_G$. This property, however, is not sufficient for proving
Conjecture \ref{eigA}. But in the case $G=PGL_n$ we can derive it from
the system of differential equations in Theorem \ref{mainthm}. We
expect that a similar argument also proves Conjecture \ref{eigA} for a
general $G$.\qed
\end{remark}

In the rest of this subsection, we will {\em assume the validity of
  Conjecture \ref{eigA}}. It implies the following statement. Let
  $$
  \mcS({\mc A}_G) := \bigoplus_{s \in
    \on{Spec}(\mh_G)}\mcH_G(s) \subset \mcH_G.
  $$
  According to Conjecture \ref{eigA}, the algebra $\mcA_G$ acts on
  it.

\begin{corollary}    \label{smooth}
The closures of the actions of the elements of $\mcA_G$ on $\mcS({\mc
  A}_G)$ form a family of commuting normal operators, and each
$\mcH_G(s)$ is an eigenspace of these normal operators.
\end{corollary}

\begin{remark} It is clear that the subspace $\mcS({\mc A}_G) \subset
  \mcH _G$ is contained in the {\it Schwartz space} $S({\mc A}_G)$
  defined in \cite{EFK}. We expect that $S({\mc A}_G)$ is a completion
  of $\mcS({\mc A}_G)$ corresponding to a specific growth condition on
  the coefficients.\qed
\end{remark}

Corollary \ref{smooth} implies that we can define the joint spectrum
of the commutative algebra ${\mc A}_G$. It follows from the definition
that ${\rm Spec}(\mcA_G )$ is naturally realized as a subset of ${\rm
  Op}^\ga_{\LG}(X) \times \ov{{\rm Op}}^\ga_{\LG}(X)$.

\begin{definition}
  {\em An $\LG$-oper $\chi$ is called an {\em oper with real
      monodromy} (or {\em real oper} for short) if the monodromy
    representation $\pi_1(X,p_0) \to \LG(\C)$ corresponding to $\chi$
    is isomorphic to its complex conjugate.}
\end{definition}
  
Denote by $\on{Op}^\ga_{\LG}(X)_{\R}$ the subset of real $\LG$-opers
in $\on{Op}^\ga_{\LG}(X)$.

\begin{remark} It follows from the above definition that
the image of the monodromy representation corresponding to a real
$\LG$-oper is contained in a real form $\LG_{\mathbb R}$ of
$\LG$. Moreover, this form is inner to the split real form since for
every algebraic representation $V$ of $\LG$, the corresponding
monodromy $\rho_V$ is isomorphic to $\overline \rho_V$. We conjecture
that in fact the form $\LG_{\mathbb R}$ is the split real form
$\LG({\mathbb R})$; in other words, that real $\LG$-opers are precisely the
$\LG$-opers $\chi$ for which the image of the monodromy representation
of $\chi$ in $\LG(\C)$ is contained, up to conjugation, in
$\LG({\mathbb R})$.  This is known to be the case for $\LG=PGL_2$ and
$SL_2$, see \cite{F,Go,GKM}.

We can prove that $\LG_{\mathbb R}$ is split in the case when $S \neq
\emptyset$. Indeed, the residue of the oper at each parabolic point is
regular nilpotent, so the monodromy is regular unipotent. Thus our
real form $\LG_{\mathbb R}$ contains a real regular unipotent element,
hence a real Borel subgroup (the unique Borel subgroup containing it),
hence it is quasi-split. But an inner quasi-split real form must be
split.\qed
\end{remark}

For $\chi \in \on{Op}^\ga_{\LG}(X)$, we denote, following Sect. 3.4 of
\cite{EFK}, by $\chi^*$ the $\LG$-oper obtained by applying to $\chi$
the Chevalley involution on $\LG$ (it belongs to the same component
$\on{Op}^\ga_{\LG}(X)$). Theorem 4.2 of \cite{EFK} implies the
following.

\begin{theorem}\label{sa1}
As a subset of ${\rm Op}^\ga_{\LG}(X) \times \ov{{\rm
    Op}}^\ga_{\LG}(X)$, the joint spectrum ${\rm Spec}(\mcA_G)$
consists of pairs $(\chi,\overline \chi^*)$ where $\chi \in
\on{Op}^\ga_{\LG}(X)_{\R}$. Thus, ${\rm Spec}(\mcA_G)$ is naturally
realized as a subset of $\on{Op}^\ga_{\LG}(X)_{\R}$.
\end{theorem}

\begin{conjecture}    \label{SpecA}
{\em (1)} The set $\on{Op}^\ga_{\LG}(X)_{\R}$ is discrete.

{\em (2)} ${\rm Spec}(\mcA_G ) = \on{Op}^\ga_{\LG}(X)_{\R}$.
\end{conjecture}

\begin{remark}
(i) In \cite{EFK}, we proved Conjecture \ref{SpecA} in the case $G=SL_2,
X=\pone, |S|=4$, directly, without relying on the Compactness
Conjecture.

(ii) Part (1) of Conjecture \ref{SpecA} is known in the case when
  $G=PGL_2$ (see \cite{F}, Sect. 7).\qed
\end{remark}

In \cite{BD}, Sect. 5.1.1, Beilinson and Drinfeld attached to every
$\chi \in \on{Op}^\ga_{\LG}(X)$ a $D$-module on $\Bun_G$, which is a
Hecke eigensheaf with respect to the flat $\LG$-bundle corresponding
to $\chi$; namely,
\begin{equation}    \label{Delta0}
  \Delta^0_\chi := K^{-1/2}_{\Bun} \otimes \Delta_\chi
\end{equation}
where $\Delta_\chi$ is given by \eqref{Deltachi} (see also \cite{EFK},
Sect. 3.2, and Remark \ref{chi1} below). The first part of the
following statement has been proved in \cite{AGKRRV} (Corollary
11.6.7). The second part has been proved in \cite{FR} (Theorem
11.2.1.2 and Remark 11.2.1.3). (In the case $G=\mathrm{PGL}_{n}$ this
also follows from the results of \cite{Gai:outline}.)

\begin{theorem}\label{RS}
  \hfill
\begin{enumerate}
\item $\Delta^0_\chi$ has regular singularities.
\item $\Delta^0_\chi$ is irreducible on each
connected component of $\Bun_G$.
\end{enumerate}
\end{theorem}

Recall that $\Bun_G$ has connected components $\Bun_G^\beta$ labeled
by $\beta \in \pi_1(G)$. Thus, we have a natural direct sum
decomposition
\begin{equation}    \label{dirsum}
\mcH_G = \bigoplus_{\beta \in \pi_1(G)} \mcH_G^\beta.
\end{equation}
We also have
$$
\mcS({\mc A}_G) = \bigoplus_{\beta \in \pi_1(G)} \mcS({\mc
  A}_G)_\beta, \qquad \mcS({\mc A}_G)_\beta := \mcS({\mc A}_G) \cap
\mcH_G^\beta.
$$
The action of ${\mc A}_G$ on $\mcS({\mc A}_G)$ preserves the direct
summands $\mcS({\mc A}_G)_\beta$. Hence we can talk about the joint
eigenvalues of ${\mc A}_G$ on each of them and the corresponding
multiplicities.

The following statement is proved in Sect. 1.5 of \cite{EFK} in the
case when $G$ is simply-connected, and the proof generalizes in a
straightforward way to the case of an arbitrary connected simple Lie
group $G$.

\begin{proposition}    \label{mone}
  Suppose that $\chi \in \on{Op}^\ga_{\LG}(X)$ corresponds to an
  element of ${\rm Spec}(\mcA_G )$ and $\Delta^0_\chi$ has regular
  singularities and is irreducible on each connected component
  $\Bun_G^\beta$. Then the multiplicity of $\chi$ in $\mcS({\mc
    A}_G)_\beta$ is equal to one for all $\beta \in \pi_1(G)$.
\end{proposition}

\subsection{The case $G=PGL_n$}    \label{intPGLn}

According to Theorem \ref{RS}, the spectrum of $\mcA_G$
is simple on each $\mcS({\mc A}_G)_\beta \subset \mcH^\beta_G, \beta
\in \pi_1(PGL_n) = \Z/n\Z$.

For $\beta \in \Z/n\Z$, let $\psi_{\chi,\beta}$ be a non-zero
generator of the one-dimensional eigenspace of $\mcA_G$ in
$\mcH_G^\beta$ corresponding to the eigenvalue $\chi$. Let
$$
{\mc E}_\chi := \on{span} \{ \psi_{\chi,\beta}, \beta \in \Z/n\Z \}.
$$

Now consider the Hecke operators $H_{\omega_1}(x), x \in X$, and the
corresponding operator-valued section $H_{\omega_1}$ of the line
bundle $\Omega_X^{-(n-1)/2}$. These operators act from $\mcH_G^\beta$ to
$\mcH_G^{\beta+1}$. Corollary \ref{smooth} then implies that
$H_{\omega_1}(x)$ and $H_{\omega_1}$ preserve the $n$-dimensional
subspace ${\mc E}_\chi \subset \mcH_G$. Moreover, we can normalize the
vectors $\psi_{\chi,\beta}$ in such a way that $\Psi^1_\chi :=
\sum_{\beta \in \Z/n\Z} \psi_{\chi,\beta}$ is an eigenvector
of $H_{\omega_1}$. Then the vectors
$$
\Psi^\ep_\chi := \sum_{\beta \in \Z/n\Z} \ep^\beta \psi_{\chi,\beta},
\qquad \ep \in \mu_n
$$
(where $\mu_n$ is the group of $n$th roots of unity) are also
eigenvectors of $H_{\omega_1}$ and hence form an eigenbasis of ${\mc
  E}_\chi$. If $\Phi_\chi$ is a section of
$\Omega_X^{-(n-1)/2}$ representing the eigenvalue of $H_{\omega_1}$ on
$\Psi^1_\chi$, then the section of $\Omega_X^{-(n-1)/2}$ representing the
eigenvalue of $H_{\omega_1}$ on $\Psi^\ep_\chi$ is $\ep \Phi_\chi$.

Thus, to each $\chi \in \on{Spec} \mcA_G$ corresponds a collection $\{
\Psi^\ep_\chi \}$ of eigenvalues of $H_{\omega_1}$, which is naturally
a $\mu_n$-torsor. We will now write a conjectural formula for these
eigenvalues.

Recall that $\chi$ is an $SL_n$-oper in $\Op^\ga_{SL_n}(X)_{\R}$. Let
$(\V_{\omega_1},\nabla_\chi)$ be the corresponding holomorphic flat
vector bundle on $X$ with $\on{det}(\V_{\omega_1},\nabla_\chi)$
identified with the trivial flat line bundle. The definition of an
oper provides us with an embedding of a line bundle
$$
\kappa_{\omega_1}: K_X^{(n-1)/2} \hookrightarrow \V_{\omega_1},
$$
and hence an embedding
$$
\wt\kappa_{\omega_1}: \OO_X \hookrightarrow \V_{\omega_1} \otimes K_X^{-(n-1)/2}.
$$
Therefore we obtain a section
$$
s_{\omega_1} := \wt\kappa_{\omega_1}(1) \in \Gamma(X,\V_{\omega_1}
\otimes K_X^{-(n-1)/2}).
$$
In the same way, we obtain a section
$$
s_{\omega_{n-1}} \in \Gamma(X,\V_{\omega_{n-1}} \otimes
K_X^{-(n-1)/2}) = \Gamma(X,\V^*_{\omega_1} \otimes
K_X^{-(n-1)/2}).
$$
By our assumption that $\chi \in \on{Op}^\ga_{SL_n}(X)_{\R}$, the
monodromy representations associated to $\chi$ and $\ol\chi$ are
isomorphic. This means that $(\V_{\omega_1},\nabla_\chi)$ and
$(\ol{\V}_{\omega_1},\ol\nabla_\chi)$ are isomorphic as $C^\infty$
flat vector bundles on $X$. Hence we have a non-degenerate 
pairing
$$
h_{\chi}(\cdot,\cdot): (\V_{\omega_1},\nabla_{\chi}) \otimes
  (\ol{\V}_{\omega_{n-1}},\ol\nabla_{\chi}) \to ({\mc
  C}^\infty_X,d)
$$
of $C^\infty$ flat vector bundles on $X$. Since
$(\V_{\omega_1},\nabla_{\chi})$ and
$(\ol{\V}_{\omega_{n-1}},\ol\nabla_{\chi})$ are associated to flat
$SL_n$-bundles, their determinants are identified with the trivial
flat line bundle. We will require that $h_{\chi}(\cdot,\cdot)$ induce
the canonical  pairing on the corresponding determinant line
bundles. The set of $h_{\chi}(\cdot,\cdot)$ normalized this way is a
$\mu_n$-torsor, which we denote by $\{ \ep
h_{\chi}(\cdot , \cdot) \}$.

\begin{conjecture}    \label{eigH}
  $$
  \{ \Phi^\ep_\chi \} = \{ \ep
  h_{\chi}(s_{\omega_1},\ol{s_{\omega_{n-1}}}) \}
  $$
  as $\mu_n$-torsors of global sections of the line bundle
  $\Omega_X^{-(n-1)/2}$ on $X$.
\end{conjecture}

We will prove a slightly weaker form of this conjecture (Corollary
\ref{eigH1}), with roots of unity $\ep$ replaced by non-zero complex
numbers, by showing that both sides satisfy the same system of
differential equations which has a unique solution up to a scalar.

To explain this, we need an alternative description of the component
$\on{Op}^\ga_{SL_n}(X)$.

Consider $n$th order differential operators $P: K_X^{-(n-1)/2} \to
K_X^{(n+1)/2}$ (here, if $n$ is even, we use our chosen square root
$K_X^{1/2}$) such that
\begin{enumerate}
\item $\on{symb}(P) \in H^0(X,\OO_X)$ is equal to $1$;
\item The operator $P - (-1)^n P^*$ has order $n-2$ (here $P^*:
  K_X^{-(n-1)/2} \to K_X^{(n+1)/2}$ is the algebraic adjoint operator,
  see \cite{BB}, Sect. 2.4).
\end{enumerate}
These operators form an affine space, which we denote by
$D^\ga_n(X)$. Locally,
$$
P = \pa_z^n + v_{n-2}(z)\pa_z^{n-2} + \ldots + v_0(z).
$$

The following statement is proved in \cite{BD:opers}, Sect. 2.8.

\begin{lemma}    \label{Dga}
  There is a bijection $\on{Op}^\ga_{SL_n}(X) \simeq D^\ga_n(X)$
  $$
  \chi \in \on{Op}^\ga_{SL_n}(X) \quad \mapsto \quad P_\chi \in
  D^\ga_n(X)
  $$
  such that the sections $s_{\omega_1} \in \Gamma(X,\V_{\omega_1}
  \otimes K_X^{-(n-1)/2})$ and $s_{\omega_{n-1}} \in \Gamma(X,\V_{\omega_{n-1}}
  \otimes K_X^{-(n-1)/2})$ satisfy
  $$
  P_\chi \cdot s_{\omega_1} = 0, \qquad P_\chi^* \cdot s_{\omega_{n-1}} = 0,
  $$
  where $P_\chi^*$ is the algebraic adjoint of $P_\chi$
  (here we use the $\D_X$-module structures
  on $\V_{\omega_1}$ and $\V_{\omega_{n-1}}$ corresponding to the
  oper connection of $\chi$).
\end{lemma}

The flat vector bundle $(\V_{\omega_1},\nabla_{\chi})$ is known to be
irreducible if $g>1$, see \cite{BD}, Sect. 3.1.5(iii). Therefore we
obtain

\begin{corollary}    \label{unique}
    $h_{\chi}(s_{\omega_1}, \ol{s_{\omega_{n-1}}})$ is a
  unique, up to a scalar, section $\Phi$ of $\Omega_X^{-(n-1)/2}$
  which is a solution of the system of differential equations
\begin{equation}    \label{diffeqs}
    P_\chi \cdot \Phi = 0, \qquad \ol{P^*_\chi} \cdot
    \Phi = 0
\end{equation}
\end{corollary}

On the other hand, let $\V_{\omega_1}^{\on{univ}}$ be the universal
vector bundle over $\on{Op}^\ga_{SL_n}(X) \times X$ with a partial
connection $\nabla^{\on{univ}}$ along $X$, such that
$$
(\V_{\omega_1}^{\on{univ}},\nabla^{\on{univ}})|_{\chi \times X} =
(\V_{\omega_1},\nabla_\chi).
$$

Let $\V_{\omega_1,X}^{\on{univ}} := \pi_*(\V_{\omega_1}^{\on{univ}})$,
where $\pi: \on{Op}^\ga_{SL_n}(X) \times X \to X$ is the projection
and $\pi_*$ is the $\OO$-module push-forward functor. The connection
$\nabla^{\on{univ}}$ makes $\V_{\omega_1,X}^{\on{univ}}$ into a left
$\D_X$-module.

Recall (see Definition \ref{beta}) that we have denoted by $D_{PGL_n}$
the algebra of global holomorphic differential operators acting on
each component $\Bun^\beta_{PGL_n}$ of $\Bun_{PGL_n}$. By Theorem
\ref{DGOp}, we have an isomorphism
$$
D_{PGL_n} \simeq \on{Fun} \on{Op}^\ga_{SL_n}(X).
$$
Hence $D_{PGL_n}$ naturally acts on $\V_{\la,X}^{\on{univ}}$ and this
action commutes with the action of $\D_X$. Using the above realization
of $PGL_n$-opers as $n$-th order differential operators,
we construct the ``universal $PGL_n$-oper'' as follows:

\begin{lemma}    \label{nsigma}
  There is a unique $n$-th order differential operator
\begin{equation}    \label{sigman}
\sigma: K_X^{-(n-1)/2} \to D_{PGL_n} \otimes K_X^{(n+1)/2}
\end{equation}
satisfying the following property: for any $\chi \in
\on{Op}^\ga_{SL_n}(X) = \on{Spec} D_{PGL_n}$, applying the
corresponding homomorphism $D_{PGL_n} \to \C$ we obtain $P_\chi$.
\end{lemma}

The following is one of the main results of this paper, which will be
proved in Section \ref{PGLn}.

\begin{theorem}    \label{mainthm}
  The Hecke operator $\wh{H}_{\omega_1}$, viewed as an operator-valued
  section of $\Omega_X^{-(n-1)/2} = K_X^{-(n-1)/2} \otimes
  \ol{K}_X^{-(n-1)/2}$, satisfies the system of differential equations
  \begin{equation}    \label{eqHom1}
    \sigma \cdot \wh{H}_{\omega_1} = 0, \qquad \ol\sigma \cdot
    \wh{H}_{\omega_1} = 0.
  \end{equation}
\end{theorem}

\begin{corollary}    \label{eigH1}
Each of the eigenvalues $\Phi^\ep_\chi$ of the Hecke operator
$H_{\omega_1}$ is equal to a scalar multiple of
$h_{\chi}(s_{\omega_1},\ol{s_{\omega_{n-1}}})$.
\end{corollary}

\noindent{\em Proof of Corollary \ref{eigH1} from Theorem
  \ref{mainthm}.} Equations
\eqref{eqHom1} imply that the eigenvalues of $H_{\omega_1}$ satisfy
equations \eqref{diffeqs}. That's because if $v$ is an eigenvector of
${\mc A}_{PGL_n}$ with the eigenvalue of $D_{PGL_n}$ corresponding to
a holomorphic $SL_n$-oper $\chi$, then according to Theorem \ref{sa1},
the eigenvalue of $\ol{D}_{PGL_n}$ on $v$ corresponds to the
anti-holomorphic $SL_n$-oper $\ol\chi^*$. Furthermore, it is clear
that the $n$th order operator $\ol{K}_X^{-(n-1)/2} \to
\ol{K}_X^{(n+1)/2}$ associated to $\ol\chi^*$ is
$\ol{P^*_\chi}$. Corollary \ref{unique} then implies Corollary
\ref{eigH1}.\qed

\medskip

Corollary \ref{eigH1} describes the eigenvalues of the Hecke operator
$H_{\omega_1}$ for $G=PGL_n$ up to scalar multiples (it is slightly
weaker than Conjecture \ref{eigH}, which describes these eigenvalues
up to multiplication by $n$th roots of unity).

There is a similar conjectural formula for the eigenvalues of
the Hecke operators $H_{\omega_i}, i=2,\ldots,n-1$,
  corresponding to the other fundamental coweights of $PGL_n$ can be
  found in a similar way. Namely, let $({\mc
    V}_{\omega_i},\nabla_{\chi,\omega_i})$ be the $i$th exterior power
  of $({\mc V}_{\omega_1},\nabla_{\chi})$. Note that it is
  dual to $({\mc V}_{\omega_{n-i}},\nabla_{\chi,\omega_{n-i}})$. The oper
    Borel reduction gives rise to a section
  $$
  s_{\omega_i} \in \Gamma(X,{\mc V}_{\omega_i} \otimes
  K_X^{-i(n-i)/2}).
  $$
If $\chi$ is a real $PGL_n$-oper, then we have a non-degenerate 
pairing
$$
h_{\chi,\omega_i}(\cdot,\cdot): (\V_{\omega_i},\nabla_{\chi,\omega_i}) \otimes
  (\ol{\V}_{\omega_{n-i}},\ol\nabla_{\chi,\omega_{n-i}}) \to ({\mc
  C}^\infty_X,d).
$$
The eigenvalue of the Hecke operator
$H_{\omega_i}$ corresponding to $\chi$ is conjecturally equal to a
scalar multiple of
\begin{equation}    \label{ith}
  h_{\chi,\omega_i}(s_{\omega_i},\ol{s_{\omega_{n-i}}}).
\end{equation}

The eigenvalues of all other Hecke operators $H_\la$ for $G=PGL_n$ can
be found from the eigenvalues of $H_{\omega_i}, i=1,\ldots,n-1$, using
Proposition \ref{algiso}.

Moreover, Theorem \ref{mainthm} also implies Conjecture
\ref{eigA}. Indeed, it can be shown that differential equations
\eqref{eqHom1} imply a suitable version of commutativity of the
algebra $\mathcal A_G$ with the Hecke operators. This will be
discussed in more detail in our next paper \cite{EFKnew}.

In Section \ref{PGLn}, we will derive Theorem \ref{mainthm} from the
theorem of Beilinson and Drinfeld \cite{BD} describing the action of
the {\em Hecke functor} ${\mathbb H}_{\omega_1}$ on the left
$D$-module $\D_{\Bun_{PGL_n}} \otimes K^{-1/2}_{\Bun}$ on
$\Bun_{PGL_n}$. We will also discuss a generalization of these results
to an arbitrary simple Lie group $G$.

This shows a deep connection between the geometric/categorical
Langlands correspondence and the analytic/function-theoretic one.

\subsection{The structure of the paper}

The paper is organized as follows. In Section \ref{GL1} we consider
the abelian case, $G=GL_1$, in which one can already see the main
ingredients of our construction in the non-abelian case. In Section
\ref{genD} we formulate some basic results on the compatibility
between the natural operations on functions (pull-back, push-forward,
and integral transforms) and the corresponding functors between
categories of twisted $D$-modules. We then use these results in the
following sections to relate the Hecke operators and Hecke functors
and derive differential equations on the Hecke operators. In Section
\ref{PGLn}, we prove that the Hecke operators satisfy the system
\eqref{eqHom1} of differential equations corresponding to the
``universal $SL_n$-oper'' (Theorem \ref{mainthm}), using a Hecke
eigensheaf property established in \cite{BD}. In Section \ref{gencase}
we formulate the analogues of Theorem \ref{mainthm} and Corollary
\ref{eigH1} describing the eigenvalues of the Hecke operators in the
case of an arbitrary simple Lie group $G$. We then outline the proof
of these results generalizing the argument we used in Section
\ref{PGLn} in the case of $PGL_n$.

\subsection{Acknowledgments} We are grateful to M. Kontsevich for
sharing his ideas on Hecke operators, and to A. Braverman,
V. Drinfeld, D. Gaiotto, D. Gaitsgory, T. Pantev, and E. Witten for
useful discussions. E.F. thanks P. Schapira for helpful discussions
and references concerning the material of Section \ref{genD}. We also
thank the referees for helpful comments.

P.E.'s research was partially supported by the NSF grant
DMS-1916120. D.K.'s research was partially supported by the
ERC grant No 669655.

\section{The abelian case}    \label{GL1}

In this section we consider the case $G=GL_1$. Though the spectral
problem here is rather simple (standard Fourier analysis on a torus),
it provides a useful illustration of our general method. Indeed, one
can already observe in it all the essential ingredients of the picture
for a general group $G$. Hence it is instructive to consider it as a
blueprint for the general construction.

We first recall the results of \cite{F:analyt,EFK} on the eigenvalues
of the corresponding Hecke operators and the global differential
operators. We will then prove that the Hecke operators satisfy a
system of differential equations, which are analogous to equations
\eqref{eqHom1}. This system will enable us to prove a relation between
the eigenvalues of the Hecke operators and the global differential
operators and lead to an alternative, and simpler, proof of the
results of \cite{F:analyt,EFK} describing these eigenvalues. The
differential equations follow from a theorem describing the image of
the sheaf $\D$ of differential operators on the Jacobian of the curve
$X$ under the action of the corresponding {\em Hecke functor}. Thus,
these differential equations link the geometric/categorical Langlands
correspondence and the analytic Langlands correspondence. A
generalization of this equation, and this link, to $G=PGL_n$ will be
presented in Section \ref{PGLn}.

\subsection{The global picture}    \label{globargGL1}

Let $G=GL_1$. In this case, the role of $Bun^\circ_G$ is played by the
Picard scheme $\Pic(X)$, which is a fine moduli space of line bundles
on $X$. The canonical line bundle on $\Pic(X)$ is isomorphic to the
trivial line bundle, and we will fix such an isomorphism. Hence we
have a positive Hermitian inner product on the space of smooth
functions on $\Pic(X)$. Let $L^2(\Pic(X))$ be its completion. This is
our Hilbert space $\mcH_{GL_1}$.  The Hecke operators on $\mcH_{GL_1}$
are easy to define as they do not involve integration. Their spectrum
was described in \cite{F:analyt}, Sect. 2 (not only for $GL_1$ but
also for an arbitrary complex torus), whereas the spectrum of the
algebra of global differential operators was described in \cite{EFK},
Sect. 5. We recall these descriptions below.

Denote by $H_p$ the Hecke operator associated to a point $p \in X$ and
the defining one-dimensional representation of $GL_1$. It acts on
$L^2(\Pic(X))$ as the pull-back with respect to the map
sending a line bundle ${\mc L} \in \Pic(X)$,
\begin{equation}    \label{Hp1}
{\mc L} \mapsto {\mc L}(p).
\end{equation}
These operators obviously commute with each other for different $p
\in X$.

Recall that $\Pic(X)$ is a disjoint union of connected components
$\Pic^n(X), n \in \Z$, labeled by the degrees of line bundles. The
Hecke operators shift the degree by 1. Let us fix a point $p_0 \in X$
once and for all. The map \eqref{Hp1} with $p=p_0$ identifies
$\Pic^n(X)$ and $\Pic^{n+1}(X)$ for all $n \in \Z$. This implies (see
\cite{F:analyt}, Sect. 2.1, for more details) that the spectral theory
of the operators $H_p, p \in X$, on $L^2(\Pic(X))$ is equivalent to
the spectral theory of the operators
$$
_{p_0}H_p := H_p \circ H_{p_0}^{-1}, \qquad p \in X,
$$
acting on
$$
\mcH^0_{GL_1} := L^2(\Pic^0(X)).
$$
Note that $_{p_0}H_{p_0}$ is the identity operator.

Next, the Hecke operators $_{p_0}H_p$ commute with the algebra ${\mc
  A}_{GL_1}$ of global differential operators on the Jacobian
$\Pic^0(X)$, which is generated by the translation vector fields (both
holomorphic and anti-holomorphic). Hence they share the same
eigenfunctions; namely, the Fourier harmonics on the Jacobian viewed
as a real torus. We will now describe a relation between their
eigenvalues following \cite{F:analyt}, Sect. 2.4 and \cite{EFK},
Sect. 5.

To describe the eigenfunctions and eigenvalues of the operators
$_{p_0}H_p$ (following \cite{F:analyt}, Sect. 2.4), recall that as a
real torus,
\begin{equation}   \label{Picident}
\Pic^0(X) \simeq H^1(X,\R)^*/H_1(X,\Z),
\end{equation}
where the embedding $H_1(X,\Z) \hookrightarrow H^1(X,\R)^*$ is defined
by sending $\beta \in H_1(X,\Z)$ to the linear functional on $H^1(X,\R)$
\begin{equation}    \label{embed}
H^1(X,\R) \; \ni \; c \quad \mapsto \quad \int_\beta c = \int_\beta
(\omega_c + \ol{\omega}_c), \qquad \omega_c \in H^0(X,\Omega^{1,0}).
\end{equation}

Given $\ga \in H^1(X,\Z)$, denote by $\varphi_\ga$ the harmonic
representative of the image of $\gamma$ in $H^1(X,\R)$. Then
\begin{equation}    \label{omegaga}
\varphi_\ga = \omega_\ga + \ol\omega_\ga, \qquad \omega_\ga \in
H^0(X,\Omega^{1,0}).
\end{equation}
Note that these are precisely the smooth one-forms on $X$ whose
integrals over all one-cycles in $X$ are integers.

Now view $\varphi_\ga$ as a linear functional on $H^1(X,\R)^*$ and set
\begin{equation}    \label{Fourier}
f_\ga = e^{2 \pi i \varphi_\ga}, \qquad \ga \in H^1(X,\Z).
\end{equation}

\begin{lemma}
For any $\ga \in H^1(X,\Z)$, $f_\ga$ is a well-defined function on
$\Pic^0(X)$ given by \eqref{Picident}.
\end{lemma}

The functions $f_\ga, \ga \in H^1(X,\Z)$, are the {\em Fourier
  harmonics} of the torus $\Pic^0(X)$. They form an orthogonal basis
of the Hilbert space $\mcH^0_{GL_1}$.

The following statement is proved in \cite{F:analyt}, Theorem 2.4.

\begin{proposition}    \label{eig H}
  The function $f_\ga, \ga \in H^1(X,\Z)$, is an eigenfunction of the
  Hecke operators $_{p_0}H_p$. The eigenvalue $F_\ga(p)$ of
  $_{p_0}H_p$ on $f_\ga$ is given by the formula
  \begin{equation}
    F_\ga(p) = \on{exp} \left( \int_{p_0}^p 2\pi
    i(\omega_\ga+\ol\omega_\ga) \right)
  \end{equation}
where the integral is taken over any path connecting $p_0$ to $p$ (the
integral does not depend on the choice of this path).
\end{proposition}

Now consider the algebra of global differential
operators on $\Pic^0(X)$,
$$
{\mc A}_{GL_1} = D_{GL_1} \otimes \ol{D}_{GL_1},
$$
where $D_{GL_1}$ (resp., $\ol{D}_{GL_1}$) is the algebra of global
holomorphic (resp., anti-holomorphic) differential operators on
$\Pic^0(X)$. Since $\Pic^0(X)$ is compact and has an abelian group
structure, we have
$$
D_{GL_1} = \on{Sym}(\Theta \Pic^0(X)), \qquad \ol{D}_{GL_1} =
\on{Sym}(\ol\Theta \Pic^0(X)),
$$
where $\Theta \Pic^0(X)$ (resp., $\ol\Theta \Pic^0(X)$) is the
commutative Lie algebra of holomorphic (resp., anti-holomorphic)
translation vector fields on $\Pic^0(X)$. We have
$$
\Theta \Pic^0(X) \simeq H^0(X,\Omega^{1,0})^*,
\qquad \ol\Theta \Pic^0(X) \simeq H^0(X,\Omega^{0,1})^*.
$$
Therefore the eigenvalues of the algebra ${\mc A}_{GL_1}$ on any
joint eigenvector are encoded by a pair $({\mb a},{\mb b})$, where
${\mb a} \in H^0(X,\Omega^{1,0})$ and ${\mb b} \in
H^0(X,\Omega^{0,1})$.

We associate to ${\mb a}$ a holomorphic $GL_1$-oper, i.e. the trivial
line bundle on $X$ with the holomorphic connection $d + {\mb a}$. And
we associate to ${\mb b}$ an anti-holomorphic $GL_1$-oper, i.e. the
trivial line bundle on $X$ with the anti-holomorphic connection $d +
{\mb b}$.

\begin{proposition}[\cite{EFK}, Theorem 5.4]    \label{eig A}
  \hfill
  \begin{enumerate}
    \item[(1)]
  The eigenvalues of ${\mc A}_{GL_1}$ on $f_\ga$ are given by
  the pairs
$$
({\mb a},{\mb b}) = (2\pi i\omega_\ga,2\pi i \,
  \ol\omega_\ga).
$$

\item[(2)] The $GL_1$-opers $d + 2\pi i\omega_\ga, \ga \in H^1(X,\Z)$,
  are all holomorphic $GL_1$-opers on $X$ with real
  monodromy (i.e. the monodromy representation takes values in
  $GL_1(\R) \subset GL_1(\C)$).

\item[(3)] The spectrum of ${\mc A}_{GL_1}$ on $\mcH^0_{GL_1}$ is in
bijection with the set of holomorphic $GL_1$-opers on $X$ with real
monodromy.
\end{enumerate}
\end{proposition}

Combining Propositions \ref{eig H} and \ref{eig A}, we obtain the
following relation between the eigenvalues of the Hecke operators and
${\mc A}_{GL_1}$.

\begin{theorem}    \label{Fp}
Let $F(p)$ and $({\mb a},{\mb b})$ be the eigenvalues of the Hecke
operators $_{p_0}H_p, p \in X$, and ${\mc A}_{GL_1}$ on a joint
eigenfunction in $L^2(\Pic^0(X))$, respectively. Then
\begin{equation}    \label{relGL1}
F(p) = \on{exp} \left( \int_{p_0}^p ({\mb a}+{\mb b}) \right)
\end{equation}
and ${\mb b} = -\ol{\mb a}$.
\end{theorem}

This is equivalent to the following statement. Denote by $\pa$ and
$\ol\pa$ the holomorphic and anti-holomorphic de Rham differentials,
respectively.

\begin{proposition}    \label{Fp1}
The function $F(p)$ is the single-valued solution of the differential
equation
\begin{equation}    \label{diffGL1}
dF = ({\mb a}+{\mb b})F
\end{equation}
or equivalently, the system
\begin{equation}    \label{diffGL1-1}
\pa F = {\mb a} F, \qquad \ol\pa F = {\mb b} F
\end{equation}
normalized so that $F(p_0)=1$.
\end{proposition}

\begin{remark}
Theorem \ref{Fp} does not specify the possible values of the
one-forms ${\mb a}$ and ${\mb b}$ that appear in this relation (which
we already know from Propositions \ref{eig H} and \ref{eig
  A}). But these values can be readily obtained from the relation
\eqref{relGL1}.

Indeed, single-valuedness of the function $F(p)$ given by
\eqref{relGL1} implies that the integrals of the one-form ${\mb a} +
{\mb b}$ over all one-cycles in $X$ are integer multiples of
$2\pi i$. This is equivalent to
\begin{equation}    \label{a+b}
{\mb a} + {\mb b} = 2\pi i \varphi_\ga,
\end{equation}
where $\varphi_\ga = \omega_\ga + \ol\omega_\ga$ is the harmonic
one-form introduced above, for some $\ga \in H^1(X,\Z)$. On the other
hand, the self-adjointness on $L^2(\Pic^0(X))$ of the operators of the
form $\xi - \ol\xi$ and $(\xi + \ol\xi)/i$, where $\xi$ is any
holomorphic translation vector field on $\Pic^0(X)$, implies that
\begin{equation}    \label{a-b}
  {\mb b} = - \ol{\mb a}
\end{equation}
Combining formulas \eqref{a+b} and \eqref{a-b}, we obtain that
$$
{\mb a} = 2\pi i \omega_\ga, \qquad {\mb b} = 2\pi i \, \ol\omega_\ga
$$
for some $\ga \in H^1(X,\Z)$.

Therefore, we can derive Propositions \ref{eig H} and \ref{eig A} from
Theorem \ref{Fp}. Thus, Theorem \ref{Fp} (or equivalently, Proposition
\ref{Fp1}) is the key statement that yields explicit formulas for the
eigenvalues of both Hecke operators and the global differential
operators.\qed
\end{remark}

In the rest of this section, we will give an alternative proof of
Proposition \ref{Fp1} using the action of the {\em Hecke functors}
$_{p_0}{\mathbf H}_p$ (categorical versions of the Hecke operators
$_{p_0}H_p$) on the sheaf $\D_{GL_1}$ of differential operators on
$\Pic^0(X)$. This will be our blueprint for proving analogous
statements for a general group $G$.

\subsection{Hecke operators and Hecke functors}

Consider the $GL_1$ version of the Hecke correspondence
\begin{equation}    \label{Hecke GL1}
{\mc Hecke} = \{ ({\mc L},{\mc L}(p_0-p),p) \} \subset
\Pic^0(X) \times \Pic^0(X) \times X
\end{equation}
and let $q_1, q_2: {\mc Hecke} \to \Pic^0(X)$ be the two projections
and $q_3: {\mc Hecke} \to X$. We have
$$
q_2 \times q_3: {\mc Hecke} \overset{\sim}\longrightarrow \Pic^0(X) \times
X
$$
and with this isomorphism, $q_1$ becomes the map
\begin{eqnarray}
  q_1: \Pic^0(X) \times X & \to & \Pic^0(X) \\ \label{q1}
  {\mc L} & \mapsto & {\mc L}(p-p_0)
\end{eqnarray}

Denote by $_{p_0}\wh{H}_p$ the restriction of $_{p_0}H_p$ to the dense
subspace of $C^\infty$ functions on $\Pic^0(X)$. This is the operator
of pulling back a $C^\infty$ function under the map $q_1$ and
restricting the result to $\Pic^0(X) \times p$. As $p$ varies along
$X$, these operators combine into a single operator
$$
_{p_0}\wh{H}: C^\infty(\Pic^0(X)) \to C^\infty(\Pic^0(X) \times X).
$$

We are going to relate it to the {\em Hecke functor}
$$
_{p_0}{\mathbf H} := q_1^*: \on{Mod}(\D_{\Pic^0(X)}) \to
\on{Mod}(\D_{\Pic^0(X) \times X}),
$$
the $D$-module pull-back functor.

\subsection{Another derivation of the differential equations}


We will now derive the differential equations \eqref{diffGL1-1}
appearing in Proposition \ref{Fp1} using Corollary \ref{from D to fun}
from Section \ref{pull}, in which we will take $Z = \Pic^0(X) \times
X, Y = \Pic^0(X)$ and let $q_1$ be the map $Z \to Y$ given by
\eqref{q1}. Denote the corresponding section $1_{Z \to Y}$ (see
Section \ref{pull}) by $1_{q_1}$. The equations \eqref{diffGL1-1} will
follow from Corollary \ref{from D to fun} and the fact that $1_{q_1}$
satisfies the differential equation of Proposition \ref{eqq1} below.

To prove the Proposition \ref{eqq1}, denote by $\D_{GL_1}$ the sheaf
of holomorphic differential operators on $\Pic^0(X)$. The algebra
$$
D_{GL_1} = \Gamma(\Pic^0(X),\D_{GL_1}) = \on{Sym}(\Theta \Pic^0(X))
$$
acts by endomorphisms of $\D_{GL_1}$ from the right, and this action
commutes with the left action of $\D_{GL_1}$. We know that
$$
\on{Spec} D_{GL_1} \simeq H^0(X,\Omega^{1,0}) =
\on{Op}_{GL_1}(X).
$$
Recall that a (holomorphic) $GL_1$-oper on $X$ is a holomorphic
connection $\nabla$ on the trivial line bundle on $X$, so we can write
\begin{equation}    \label{nablaa}
\nabla = \nabla_{\mb a} := d + {\mb a},
\end{equation}
where ${\mb a} \in H^0(X,\Omega^{1,0})$.

The trivial line bundle $\V$ on $\on{Op}_{GL_1}(X) \times X$ is
equipped with a partial connection $\nabla^{\on{univ}}$ along $X$,
whose restriction to $\nabla_{\mb a} \times X \subset
\on{Op}_{GL_1}(X) \times X$ is the connection $\nabla_{\mb a}$ on
$X$. Thus, $(\V,\nabla^{\on{univ}})$ is the universal $GL_1$-oper on
$\on{Op}_{GL_1}(X) \times X$. We can write an explicit formula for
$\nabla^{\on{univ}}$.

Let $W$ be a finite-dimensional vector space and $A = \on{Fun} W =
\on{Sym} W^*$. The canonical element of $W^* \otimes W$ gives rise to
an element of $A \otimes W$. Taking $W = H^0(X,\Omega^{1,0})$, we
obtain a holomorphic one-form $\sigma$ with values in $\on{Fun}
\on{Op}_{GL_1}(X)$. Then
\begin{equation}    \label{nabla-GL1}
  \nabla^{\on{univ}} = d + \sigma.
\end{equation}

Explicitly, if $\{ \omega_i, i=1,\ldots,g \}$ is a
basis of $H^0(X,\Omega^{1,0})$ and $\{ b_i, i=1,\ldots,g \}$ is the
dual basis in $H^0(X,\Omega^{1,0})^* \simeq \Theta \Pic^0(X)$, then
\begin{equation}    \label{sigma-GL1}
\sigma = \sum_{i=1}^g b_i \omega_i.
\end{equation}

Let $\V_X := \pi_*(\V)$, where $\pi$ is the projection
$\on{Op}_{GL_1}(X) \times X \to X$. The connection
$\nabla^{\on{univ}}$ makes $\V_X$ into a left $\D_X$-module.

Moreover, the unit $1 \in \on{Fun} \on{Op}_{GL_1}(X)$ gives rise to a
global section of $\V_X$, which we denote by $1_{\V_X}$. In addition,
$\V_X$ is equipped with an action of $\on{Fun} \on{Op}_{GL_1}(X)
\simeq D_{GL_1}$ which commutes with the action of $\D_X$.

This allows us to define the following $D$-module on $\Pic^0(X) \times
X$:
$$
\D_{GL_1} \underset{D_{GL_1}}\boxtimes \V_X.
$$
The algebra $D_{GL_1}$ acts on it by endomorphisms which commute with
the action of the sheaf $\D_{\Pic^0(X) \times X}$.

Define its global section $s$ by the formula
$$
s := 1 \boxtimes 1_{\V_X}.
$$
The element \eqref{sigma-GL1} gives rise to a linear operator
$$
\sigma: \D_{GL_1}
\underset{D_{GL_1}}\boxtimes \V_X \to \D_{GL_1}
\underset{D_{GL_1}}\boxtimes (\V_X \otimes K_X).
$$
(namely, we interpret the $b_i$ as elements of $\on{Sym}(\Theta
\Pic^0(X)) = D_{GL_1}$).

The $\D_X$-module structure on $\D_{GL_1} \underset{D_{GL_1}}\boxtimes
\V_X$ allows us to define the action of the holomorphic de Rham
differential $\pa_X$ along $X$ on sections of $\D_{GL_1}
\underset{D_{GL_1}}\boxtimes \V_X$,
$$
\pa_X: \D_{GL_1}
\underset{D_{GL_1}}\boxtimes \V_X \to \D_{GL_1}
\underset{D_{GL_1}}\boxtimes (\V_X \otimes K_X).
$$
Formula \eqref{nabla-GL1} readily implies the following.

\begin{lemma}    \label{seq}
  The section $s$ satisfies the equation
\begin{equation}    \label{seq1}
  \pa_X s = \sigma \cdot s.
\end{equation}
\end{lemma}

\begin{corollary}    \label{modified}
The $D$-module $\D_{GL_1} \underset{D_{GL_1}}\boxtimes \V_X$ is
isomorphic to $\D_{GL_1} \boxtimes \OO_X$ with the action of $\D_X$ on
the second factor modified so that the holomorphic de Rham
differential $\pa_X$ acts as follows
\begin{equation}    \label{paX}
  \pa_X \mapsto 1 \boxtimes \pa_X + \sigma
\end{equation}
where $\sigma \mapsto \sum_{i=1}^g b_i \boxtimes \omega_i$.
\end{corollary}

Now recall the $D$-module $q_1^*(\D_X)$ on $\Pic^0(X) \times X$ and
its section $1_{q_1}$. The algebra $D_{GL_1}$ naturally acts on
$q_1^*(\D_X)$ by endomorphisms commuting with the $D$-module
structure.

\begin{theorem}    \label{Hecke-GL1}
There is an isomorphism of $D$-modules on $\Pic^0(X) \times X$
equipped with a commuting action of $D_{GL_1}$,
\begin{equation}    \label{eigen-GL1}
  q_1^*(\D_{GL_1}) \simeq \D_{GL_1} \underset{D_{GL_1}}\boxtimes \V_X,
\end{equation}
under which the section $1_{q_1}$ is mapped to $s$.
\end{theorem}

\begin{proof}
By Corollary \ref{modified}, we need to prove that $q_1^*(\D_{GL_1})
\simeq \D_{GL_1} \boxtimes \OO_X$, with the modified action of $\D_X$
on the second factor.

Let $A$ be an abelian variety. It comes equipped with the
group homomorphism $m: A \times A \to A$. We claim that $m^*(\D_A)
\simeq \D_A \boxtimes \OO_A$, with $\D_A$ corresponding to the first
factor acting on itself from the left, and $\D_A$ corresponding to the
second factor acting on $\OO_A$ in a modified way so that the
corresponding action of the holomorphic de Rham differential is given
by the formula
\begin{equation}    \label{paA}
\pa_A \mapsto 1 \boxtimes \pa_A + \sigma_A,
\end{equation}
where
$$
\sigma_A = \sum_i b_i \boxtimes \omega_i.
$$
Here $\{ \omega_i \}$ is a basis in the space $T^*_1 A = H^0(A,K_A)$
and $\{ b_i \}$ is the dual basis in the dual vector space $T_1 A$,
which we identify with the space $\Theta(A)$ of holomorphic
translation vector fields on $A$ (so the $b_i$ can be viewed as global
differential operators on $A$).

To see this, we invoke formula \eqref{Z to Y} from the next section
which shows that
\begin{equation}    \label{DA}
m^*(\D_A) = \D_{A \times A}/(\D_{A \times A} \cdot \Theta_{A \times A/A})
\end{equation}
where $\Theta_{A \times A/A}$ is the sheaf of vertical vector fields
for the morphism $m: A \times A \to A$. It is generated by the global
translation vector fields on $A \times A$ of the form $\xi \boxtimes 1
- 1 \boxtimes \xi$, where $\xi \in \Theta(A)$. The quotient \eqref{DA}
can be identified with $\D_A \boxtimes \OO_A$ but then the action of
the vector field $1 \boxtimes \xi$ corresponding to the element $\xi$
in the second factor $\D_A$ is given by the action of $\xi \boxtimes
1$ corresponding to the the element $\xi$ in first factor
$\D_A$, which coincides with the above description.

Now we apply this to the case $A = \Pic^0(X)$ (so $\D_A = \D_{GL_1}$)
and observe that
$$
q_1^*(\D_{GL_1}) \simeq (m^*(\D_{GL_1}))|_{\Pic^0(X) \times X}
$$
where $m$ is the multiplication map on $\Pic^0(X)$ and $X$ is embedded
into the second factor $\Pic^0(X)$ via the Abel--Jacobi map $X \to
\Pic^0(X)$ sending $p \mapsto \OO_X(p-p_0)$.

The theorem now follows from the fact that the linear map from
$$
H^0(\Pic^0(X),K_{\Pic^0(X)}) \simeq H^0(X,K_X)
$$
to $H^0(X,K_X)$ corresponding to the pull-back of a one-form under the
Abel-Jacobi map is the identity, so that formula \eqref{paA} becomes
\eqref{paX}.
\end{proof}

Using Theorem \ref{Hecke-GL1} and Lemma \ref{seq}, we obtain the
sought-after differential equation on $1_{q_1}$.

\begin{proposition}    \label{eqq1}
  $\pa_X 1_{q_1} = \sigma \cdot 1_{q_1}.$
\end{proposition}

Corollary \ref{from D to fun}, which is proved in Section \ref{pull}
below, gives us a way to produce differential equations on the
operator of pull-back of functions from differential equations
satisfied by the unit section of the $\D$-module pull-back of the
sheaf of differential operators. More precisely, we will use Corollary
\ref{from D to fun} in the setting where the map $p: Z \to Y$ is $q_1:
\Pic^0(X) \times X \to \Pic^0(X)$, so that ${\mb p}\ud_{C^\infty} =
    {}_{p_0}\wh{H}$ and $1_{Z \to Y} = 1_{q_1}$. Combined with Proposition
    \ref{eqq1}, Corollary \ref{from D to fun} gives us the main result
    of this section:

\begin{theorem}
  The Hecke operator $_{p_0}\wh{H}$, viewed as a smooth function on $X$
  with values in operators on $C^\infty(\Pic^0(X))$, satisfies the
  system of differential equations
\begin{equation}    \label{HGL1}
\pa_X \cdot {}_{p_0}\wh{H} = \sigma \cdot  {}_{p_0}\wh{H}, \qquad
\ol\pa_X \cdot {}_{p_0}\wh{H} = \ol\sigma \cdot  {}_{p_0}\wh{H}.
\end{equation}
\end{theorem}

This theorem implies Proposition \ref{Fp1}. Thus, we have obtained an
alternative proof of Proposition \ref{Fp1} which relies on the Hecke
eigensheaf property (Theorem \ref{Hecke-GL1}) of the sheaf $\D_{GL_1}$
(in the sense of Remark \ref{chii} below).

In other words, the differential equation on the eigenvalues of the
{\em Hecke operators} follows from the description of the action of
the {\em Hecke functor} on the sheaf $\D_{GL_1}$.

\begin{remark}
The equations \eqref{HGL1} yield the following explicit formula for
the Hecke operators in terms of the translation vector fields:
\begin{equation}    \label{Hp}
_{p_0} H_p = \on{exp} \left( \int_{p_0}^p (\sigma +
  \ol\sigma) \right) = \on{exp} \left( \int_{p_0}^p \left(
  \sum_{i=1}^g b_i \omega_i + \sum_{i=1}^g \ol{b}_i \ol\omega_i
  \right) \right)
\end{equation}
where the $b_i$ and $\ol{b}_i$ are viewed as translation vector fields
on $\Pic^0(X)$ (see formula \eqref{sigma-GL1}), whereas the $\omega_i$
and $\ol\omega_i$ are one-forms on $X$ which are integrated over a
path connecting the points $p_0$ and $p$. Integrality properties of
the eigenvalues of the $b_i$ and $\ol{b}_i$ then imply that the
exponential of this integral does not depend on the choice of such
path. This way we obtain another interpretation of the relation
between the eigenvalues from Theorem \ref{Fp}. However, the
path-independence of \eqref{Hp} is not obvious, and this is why we
prefer to express this relation in terms of the system
\eqref{HGL1}.\qed
\end{remark}

\begin{remark}    \label{chii}
  For ${\mb a} \in \on{Op}_{GL_1}(X) = \on{Spec} D_{GL_1}$, let
  $I_{\mb a}$ be the corresponding ideal in $D_{GL_1}$. Define the
  following $D$-module on $\Pic^0(X)$:
$$
  \Delta_{\mb a} := \D_{GL_1}/(\D_{GL_1} \cdot I_{\mb a}).
$$
  Theorem \ref{Hecke-GL1} implies that
$$
  q_1^*(\Delta_{\mb a}) \simeq \Delta_{\mb a} \boxtimes \V_{\mb a}
  $$
  where $\V_{\mb a} = (\OO_X,\nabla_{\mb a})$ is the flat line bundle
  corresponding to ${\mb a}$ (see formula \eqref{nablaa}). The last
  isomorphism expresses the fact that $\Delta_{\mb a}$ is the
  restriction to $\Pic^0(X)$ of a {\em Hecke eigensheaf} on $\Pic(X)$
  with the eigenvalue $\V_{\mb a}$. In fact, Theorem \ref{Hecke-GL1} is
  equivalent to this statement for all ${\mb a} \in
  \on{Op}_{GL_1}(X)$. In this sense, $\D_{GL_1}$ is the universal
  Hecke eigesheaf parametrized by all $GL_1$-opers on $X$.

  Here's a categorical interpretation. Recall that a $D$-module
  version of the Fourier--Mukai transform \cite{L,R} establishes an
  equivalence between the category of coherent $D$-modules on
  $\Pic^0(X)$ and the category of coherent sheaves on
  $\on{Loc}_{GL_1}$, the moduli space of flat line bundles on $X$ (it
  is isomorphic to an affine bundle over $\Pic^0(X)$). The space
  $\on{Op}_{GL_1}(X)$ can be a realized as a subvariety of
  $\on{Loc}_{GL_1}$ (it is the fiber over the point corresponding to
  the trivial line bundle in $\Pic^0(X)$). Let
  $\on{Coh}(\on{Op}_{GL_1}(X))$ be the category of coherent sheaves
  supported (scheme-theoretically) on $\on{Op}_{GL_1}(X) \subset
  \on{Loc}_{GL_1}$. The restriction of the above Fourier--Mukai
  transform to $\on{Coh}(\on{Op}_{GL_1}(X))$ gives rise to an
  equivalence $E$ between this (abelian) category and the category of
  $D$-modules ${\mc K}$ on $\Pic^0(X)$ with finite global
  presentation, i.e. such that there is an exact sequence
      $$
      \D_{GL_1}^{\oplus m} \to \D_{GL_1}^{\oplus n} \to
             {\mc K}
      \to 0.
      $$
  This equivalence $E$ takes an object ${\mc F}$ of
  $\on{Coh}(\on{Op}_{GL_1}(X))$ to
    $$
    E({\mc F}) := \D_{GL_1} \underset{D_{GL_1}}\otimes
    F, \qquad F := \Gamma(\on{Op}_{GL_1}(X),{\mc F}),
      $$
      where we use the fact that $D_{GL_1} \simeq \on{Fun}
      \on{Op}_{GL_1}(X)$ (see \cite{FT}, Sect. 2).

      It follows from this definition that $E(\OO_{\mb a}) =
      \Delta_{\mb a}$ and $E(\OO_{\on{Op}_{GL_1}(X)}) = \D_{GL_1}$.
      \qed
\end{remark}

\section{Generalities on $D$-modules and integral
  transforms}    \label{genD}

In this section we discuss the compatibility between natural functors
on the categories of (twisted) $\D$-modules and the corresponding
operations on sections of line bundles. Though most of the results of
this section are fairly strightforward, we were unable to find them in
the literature. We expect that these results are likely to have other
applications, so it is worthwhile to record them here.

We note that closely related topics have been discussed in the works
by A. D'Agnolo and P. Schapira \cite{DS} and A. Goncharov \cite{Gon},
and in fact, we will use a result of \cite{DS} below. Also, in writing
this section we have benefited from the advice of P. Schapira.

First, we discuss the pull-back functor (Section \ref{pull}), then the
pull-back functor in the setting of twisted differential operators
(Section \ref{pull1}), then the push-forward functor (Section
\ref{push}), and finally the integral transform functors associated to
correspondences (Section \ref{inttr}). We will use the
  results of the last subsection (specifically, Corollary
  \ref{diffcorr1}) to prove Theorem \ref{mainthm} in Section
  \ref{PGLn}. Roughly speaking, the differential equations
  \eqref{eqHom1} on the Hecke operator will follow from certain
  properties of the corresponding Hecke functor established in
  \cite{BD}.

\begin{remark}    \label{Grothendieck}
In the case of a curve over a finite field $\Fq$, one can pass from
Hecke eigensheaves to Hecke eigenfunctions using Grothendieck's {\em
  faisceaux-fonctions} correspondence. Namely, taking the trace of the
Frobenius (a topological generator of the Galois group of $\Fq$) on
the stalks of a Hecke eigensheaf on $\Bun_G$, we obtain a function on
$\Bun_G(\Fq)$. Crucially, this function is a Hecke eigenfunction
because Grothendieck's correspondence is compatible with natural
operations on sheaves and functions.

In the case of a curve defined over $\C$, there is no Frobenius as the
field of complex numbers is algebraically closed. Instead, we
construct Hecke eigenfunctions as single-valued bilinear combinations
of sections of the corresponding Hecke eigensheaf $\Delta_\chi$ (see
formula \eqref{Deltachi}) and sections of a complex conjugate sheaf,
as explained in Section 1.5 of \cite{EFK} and Section \ref{caseC}
above. One could argue that this procedure is what replaces taking the
traces of the Frobenius in the case of a curve over $\C$, but the
question remains why the resulting function is an eigenfunction of the
Hecke operators. The answer is that the results of this section enable
us to derive the Hecke eigenfunction property (and to compute the
corresponding eigenvalues) from the Hecke eigensheaf property of
$\Delta_\chi$, using the cyclicity of the Hecke eigensheaf
$\Delta_\chi$ (viewed as a twisted $D$-module on $\Bun_G$) and the
cyclicity of the corresponding ``Hecke eigenvalues'' (viewed as
twisted $D$-modules on the curve $X$). Thus, the results established
in this section may be viewed as an analogue in the complex case of
the compatibility of Grothendieck's correspondence with natural
operations on sheaves and functions.\qed
\end{remark}

\subsection{Pull-back}    \label{pull}

For a smooth complex manifold $Y$, denote by $\OO_Y$ and $\D_Y$ the
sheaves of holomorphic functions and differential operators on $Y$,
respectively (in the analytic topology). Let $\on{Mod}(\D_Y)$ be the
category of left $\D_Y$-modules.

\begin{remark}
  In what follows, we can take as $\D_Y$ the sheaf of algebraic
  differential operators on $Y$ and the category of modules over
  it. Then we have analogous statements as well.\qed
\end{remark}

Given sheaves of algebras $\mcA$ and $\mcB$ on $Y$, an
$(\mcA,\mcB)$-{\em bimodule} is, by definition, a sheaf of modules
over $\mcA \otimes \mcB^{\on{opp}}$ on $Y$.

We need to recall some facts about the pull-back functor for
$D$-modules. Let $p: Z \to Y$ be a submersion of smooth complex
manifolds. Denote by $p\ud$ the sheaf-theoretic pull-back functor.

The $\OO$-module pull-back functor $p^*$ is defined as follows. If
$\mcF$ is an $\OO_Y$-module, then the $\OO_Z$-module $p^*(\mcF)$ is
$$
p^*(\mcF) := \OO_Z \underset{p\ud(\OO_Y)}\otimes p\ud(\mcF).
$$

If $\mcF$ is a left $\D_Y$-module, then $p^*(\mcF)$ has a natural
structure of left $\D_Z$-module. To explain this, let
$\Theta_{Z/Y}$ be the relative tangent sheaf of the morphism $p: Z \to
Y$ (its sections are vertical vector fields on $Z$ with respect to
$p$). This is a sheaf of Lie algebras. It acts by commutator on
$\D_Z$, and this action preserves the left ideal $(\D_Z \cdot
\Theta_{Z/Y}) \subset \D_Z$. Moreover, we have
\begin{equation}    \label{DY1}
    (\D_Z/(\D_Z \cdot \Theta_{Z/Y}))^{\Theta_{Z/Y}} \simeq
      p\ud(\D_Y)
\end{equation}
and hence
\begin{equation}    \label{Z to Y}
    \D_{Z \to Y} := \D_Z/(\D_Z \cdot \Theta_{Z/Y})
\end{equation}
is naturally a $(\D_Z,p\ud(\D_Y))$-bimodule.

\begin{remark}
Since $p$ is a submersion,
\begin{equation}    \label{ZtoY}
\D_{Z \to Y} \simeq p^*(\D_Y) = \OO_Z \underset{p\ud(\OO_Y)}\otimes
p\ud(\D_Y)
\end{equation}
as an $(\OO_Z,p\ud(\D_Y))$-bimodule, and this is how $\D_{Z \to Y}$ is
usually defined. Defining it by formula \eqref{Z to Y} for such
$p$ has the advantage that it makes the $\D_Z$-module structure on it
manifest.\qed
\end{remark}

The isomorphism \eqref{ZtoY} implies that the $\OO$-module pull-back
$p^*(\mcF)$ of a left $\D_Y$-module $\mcF$ can be written as
$$
p^*(\mcF) \simeq \D_{Z \to Y} \underset{p\ud(\D_Y)}\otimes p\ud(\mcF)
$$
and hence the $\OO_Z$-module structure on $p^*(\mcF)$ naturally
extends to a $\D_Z$-module structure. Thus, we obtain the pull-back
functor for $D$-modules, which we denote in the same way as the
$\OO$-module pull-back:
$$
p^*: \on{Mod}(\D_Y) \to \on{Mod}(\D_Z).
$$

On the other hand, let ${\mc C}^\infty_Y$ and ${\mc C}^\infty_Z$ be
the sheaves of $\C$-valued $C^\infty$ functions on the smooth real
manifolds underlying $Y$ and $Z$, respectively (in the analytic
topology). Consider the internal Hom sheaf $${\mc Hom}(p\ud({\mc
  C}^\infty_Y),{\mc C}^\infty_Z)$$ on $Z$ in the category of sheaves
of vector spaces. By definition, for an open subset $U
\subset Z$,
\begin{equation}    \label{cont}
{\mc Hom}(p\ud({\mc C}^\infty_Y),{\mc
  C}^\infty_Z)(U) := \on{Hom}(p\ud({\mc C}^\infty_Y)|_U,{\mc
  C}^\infty_Z|_U)
\end{equation}
(as well-known, the presheaf defined by this formula is a sheaf). In
particular, we have a special global section ${\mb p}\ud_{C^\infty}$
of ${\mc Hom}(p\ud({\mc C}^\infty_Y),{\mc C}^\infty_Z)$ over $Z$,
which corresponds to the pull-back of smooth functions from $Y$ to
$Z$. More precisely, for any open $U \subset Z$, the restriction of
${\mb p}\ud_{C^\infty}$ to $U$ maps $f \in (p\ud({\mc C}^\infty_Y))(U)
= C^\infty(p(U))$ to its pull-back to $U$ via $p$ (note that $p(U)$ is
open because $p$ is a submersion).

\begin{lemma}    \label{action1}
The sheaf ${\mc Hom}(p\ud({\mc C}^\infty_Y),{\mc C}^\infty_Z)$ has a
natural $(\D_Z,p\ud(\D_Y))$-bimodule structure, which is defined as
follows: for every open $U \subset Z$, given $P \in \D_Z(U), Q \in
(p\ud(\D_Y))(U) = \D_Y(p(U))$, and $\phi \in {\mc Hom}(p\ud({\mc
  C}^\infty_Y),{\mc C}^\infty_Z)(U)$, which is a compatible system $\{
\phi_{U'} \in \on{Hom}(C^\infty(p(U')),C^\infty(U')) \}$ for open
subsets $U' \subset U$,
$$
\phi \mapsto P \circ \phi \circ Q,
$$
where $P \circ \phi \circ Q$ stands for the compatible system $\{
P|_{U'} \circ \phi_{U'} \circ Q|_{p(U')} \, | \, U' \subset U\}$.
\end{lemma}

On the other hand, we have the $(\D_Z,p\ud(\D_Y))$-bimodule $\D_{Z \to
  Y}$ given by \eqref{Z to Y}. The unit $1_Y \in \D_Y$ gives rise to a
global section $1_{Z \to Y}$ of $\D_{Z \to Y}$.

\begin{proposition}    \label{inj1}
  There is a unique injective homomorphism of
  $(\D_Z,p\ud(\D_Y))$-bimodules
  $$
  \D_{Z \to Y} \to {\mc Hom}(p\ud({\mc C}^\infty_Y),{\mc
  C}^\infty_Z)
  $$
  sending $1_{Z \to Y}$ to ${\mb p}\ud_{C^\infty}$.
\end{proposition}

\begin{proof}
Let
$$
{\mc A}_{Z,Y} := \D_Z \cdot {\mb p}\ud_{C^\infty}.
$$
For every open $U \subset Z$, we have
$$
\xi \cdot {\mb p}\ud_{C^\infty}(f) = 0, \qquad \forall \xi \in
\Theta_{Z/Y}(U), \quad \forall f \in C^\infty(p(U)).
$$
It follows that ${\mc A}_{Z,Y}$ is naturally isomorphic to the right
hand side of \eqref{Z to Y} and hence to $\D_{Z \to Y}$ as a left
$\D_Z$-module. The isomorphism \eqref{DY1} defines the structure of
a $(\D_Z,p\ud(\D_Y))$-bimodule on both ${\mc A}_{Z,Y}$ and $\D_{Z \to
  Y}$. It is clear that the former is compatible with the
$(\D_Z,p\ud(\D_Y))$-bimodule structure on ${\mc Hom}(p\ud({\mc
  C}^\infty_Y),{\mc C}^\infty_Z)$.
\end{proof}

\begin{remark} Concretely, we can choose a sufficiently fine open
  covering of $Z$, so that each neighborhood is isomorphic to the
  product of two balls $B_m \in \C^m$ and $B_n \in \C^n$ and the
  restriction of the map $p$ to it is isomorphic to the projection
  $B_m \times B_n \to B_m$. Let $\{ y_1,\ldots,y_m \}$ and $\{
  x_1,\ldots,x_n \}$ be coordinates on $B_m$ and $B_n$. It is clear
  that the spaces of sections of both ${\mc A}_{Z,Y}$ and $\D_{Z \to
    Y}$ are both isomorphic to ${\mc
    Hol}(y_1,\ldots,y_m,x_1,\ldots,x_n) \otimes
  \C[\pa_{y_1},\ldots,\pa_{y_m}]$, where ${\mc
    Hol}(y_1,\ldots,y_m,x_1,\ldots,x_n)$ denotes the space of
  holomorphic functions on $B_m \times B_n$. \qed
\end{remark}

\begin{corollary}    \label{from D to fun}
  Suppose that $P \cdot 1_{Z \to Y} = 0$ for some $P \in
  \Gamma(Z,\D_Z)$. Then
\begin{equation}
  P \cdot {\mb p}\ud_{C^\infty} = 0, \qquad \ol{P} \cdot
  {\mb p}\ud_{C^\infty} = 0.
\end{equation}
\end{corollary}

\subsection{Pull-back in the twisted setting}    \label{pull1}

The results of Section \ref{pull} can be generalized to twisted
differential operators. Namely, let $\mcL$ be a holomorphic line
bundle on $Y$. Recall that in Section \ref{mainobj}, using the norm
map $a\mapsto \norm{a} = |a|^2$ from $\C^\times$ to $\R_{>0}$, we
associated to $\mcL$ a $C^\infty$ complex line bundle
$$
\norm{\mcL} = |\mcL|^2
$$
with the structure group $\R_{>0}$ on $Y$, viewed as a complex
analytic variety. Clearly,
\begin{equation}    \label{mcL}
|\mcL|^2 \simeq (\mcL \otimes \ol\mcL)  \underset{\OO_Y \otimes
  \ol{\OO}_Y}\otimes {\mc C}^\infty_Y.
\end{equation}
In other words, if the transition functions of $\mcL$ are $\{
g_{\alpha\beta} \}$, then the transition functions of $|\mcL|^2$
are $\{ |g_{\alpha\beta}|^2 \}$. 

Consider the sheaf of (twisted) differential operators acting on
$\mcL$ (see \cite{BB}),
$$
\D_{Y,\mcL} := \mcL \underset{\OO_Y}\otimes \D_Y
\underset{\OO_Y}\otimes \mcL^{-1}.
$$
Likewise, we have the
sheaf $\D_{Z,p^*{\mcL}}$ of differential operators acting on the line
bundle $p^*(\mcL)$ on $Z$.

The role of $\D_{Z \to Y}$ is now played by
\begin{equation}    \label{role}
\D_{Z \to Y,\mcL} := p^*({\mcL}) \underset{\OO_Z}\otimes p^*(\D_Y
\underset{\OO_Y}\otimes \mcL^{-1}) \simeq p^*(\mcL)
\underset{\OO_Z}\otimes \D_{Z \to Y} \underset{p^{-1}(\OO_Y)}\otimes
p^{-1}(\mcL),
\end{equation}
which is naturally a
$(\D_{Z,p^*(\mcL)},p\ud(\D_{Y,\mcL}))$-bimodule. The unit $1_Y \in
\D_Y$ gives rise to a global section $1_{Z \to Y,\mcL}$ of $\D_{Z \to
  Y,\mcL}$.

On the other hand, let ${\mc C}^\infty_{Y,|\mcL|^2}$ and ${\mc
  C}^\infty_{Z,|p^*(\mcL)|^2}$ be the sheaves of $C^\infty$ sections of
the line bundles $|\mcL|^2$ on $Y$ and $|p^*(\mcL)|^2$ on $Z$,
respectively. Then $\D_{Y,\mcL}$ naturally acts on ${\mc
    C}^\infty_{Y,|\mcL|^2}$ and $\D_{Z,p^*(\mcL)}$ naturally acts on ${\mc
    C}^\infty_{Z,|p^*(\mcL)|^2}$.

Consider the internal Hom sheaf ${\mc
  Hom}(p\ud({\mc C}^\infty_{Y,|\mcL|^2}),{\mc
  C}^\infty_{Z,|p^*(\mcL)|^2})$ on $Z$ and its global section
${\mb p}\ud_{\mcL}$ corresponding to the natural pull-back map
\begin{equation}    \label{pullback1}
{\mb p}\ud_{\mcL}: C^\infty(Y,|\mcL|^2) \to C^\infty(Z,|p^*(\mcL)|^2).
\end{equation}
The sheaf ${\mc Hom}(p\ud({\mc C}^\infty_{Y,|\mcL|^2}),{\mc
  C}^\infty_{Z,|p^*(\mcL)|^2})$ has the structure of a
$(\D_{Z,p^*(\mcL)},p\ud(\D_{Y,\mcL}))$-bimo\-dule defined in the same
way as in Lemma \ref{action1}.

\begin{proposition}    \label{inj2}
  There is a unique injective homomorphism of
  $(\D_{Z,p^*(\mcL)},p\ud(\D_{Y,\mcL}))$-bimo\-dules
  $$
  \D_{Z \to Y,\mcL} \to {\mc Hom}(p\ud({\mc C}^\infty_{Y,|\mcL|^2}),{\mc
  C}^\infty_{Z,|p^*(\mcL)|^2})
  $$
  sending $1_{Z \to Y,\mcL}$ to ${\mb p}\ud_{\mcL}$.
\end{proposition}

\begin{proof}
According to the definition \eqref{role}, we have
\begin{equation}    \label{DZY}
\D_{Z \to Y,\mcL} \simeq p^*(\mcL) \underset{\OO_Z}\otimes \D_{Z \to Y}
\underset{p^{-1}(\OO_Y)}\otimes p^{-1}(\mcL).
\end{equation}
Moreover, isomorphisms \eqref{Z to Y} and \eqref{DY1} imply similar
isomorphisms in the twisted case (note that $\Theta_{Z/Y}$ naturally
embeds into $\D_{Z,p^*(\mcL)}$):
\begin{equation}    \label{Z to Y1}
    \D_{Z \to Y,\mcL} \simeq \D_{Z,p^*(\mcL)}/(\D_{Z,p^*(\mcL)} \cdot
    \Theta_{Z/Y}),
\end{equation}
\begin{equation}    \label{DY2}
p\ud(\D_{Y,\mcL}) \simeq (\D_{Z,p^*(\mcL)}/(\D_{Z,p^*(\mcL)} \cdot
\Theta_{Z/Y}))^{\Theta_{Z/Y}}.
\end{equation}

Now we argue in the same way as in the proof of Proposition
\ref{inj1}. Let
$$
{\mc A}_{Z,Y,\mcL} := \D_{Z,p^*(\mcL)} \cdot {\mb p}\ud_{C^\infty}.
$$
For every open $U \subset Z$, we have
$$
\xi \cdot {\mb p}\ud_{\mcL}(f) = 0, \qquad \forall \xi \in
\Theta_{Z/Y}(U), \quad \forall f \in C^\infty(p(U),\mcL).
$$
By restricting to sufficiently small open subsets, we obtain from
Proposition \ref{inj1} that ${\mc A}_{Z,Y,\mcL}$ is naturally
isomorphic to the right hand side of \eqref{Z to Y1} and hence to
$\D_{Z \to Y,\mcL}$ as a left $\D_{Z,p^*(\mcL)}$-module. Using the
isomorphism \eqref{DY2}, we obtain a
$(\D_{Z,p^*(\mcL)},p\ud(\D_{Y,\mcL}))$-bimodule structure on both
${\mc A}_{Z,Y,\mcL}$ and $\D_{Z \to Y,\mcL}$ and therefore an
isomorphism between ${\mc A}_{Z,Y,\mcL}$ and $\D_{Z \to Y,\mcL}$ as
$(\D_{Z,p^*(\mcL)},p\ud(\D_{Y,\mcL}))$-bimodules, sending $1_{Z \to
  Y,\mcL}$ to ${\mb p}\ud_{\mcL}$. By construction, the former is
precisely the $(\D_{Z,p^*(\mcL)},p\ud(\D_{Y,\mcL}))$-submodule of
\linebreak ${\mc Hom}(p\ud({\mc C}^\infty_{Y,|\mcL|^2}),{\mc
  C}^\infty_{Z,|p^*(\mcL)|^2})$ generated by ${\mb p}\ud_{\mcL}$.
\end{proof}

\subsection{Push-forward}    \label{push}

Suppose that $p: Z \to Y$ is a submersion with compact fibers, and
denote by $K_{Z/Y}$ the corresponding relative canonical
bundle. Then
$$
\Omega_{Z/Y} := |K_{Z/Y}|^2
$$
is the $C^\infty$ line bundle of relative densities. Let $K_Y$ and
$K_Z$ be the canonical line bundles on $Y$ and $Z$, respectively, and
$\Omega_Y := |K_Y|^2$ and $\Omega_Z := |K_Z|^2$ the $C^\infty$ line
bundles of densities on $Y$ and $Z$, respectively. We have
$$
\Omega_{Z/Y} \simeq \Omega_Z \underset{\OO_Z}\otimes
p^*(\Omega_Y^{-1}).
$$

Denote by $p\ld$ the sheaf-theoretic push-forward functor. Let again
$\mcL$ be a line bundle on $Y$. Consider the internal Hom sheaf (where
$\otimes$ stands for $\underset{{\mc C}^\infty_Y}\otimes$ or
$\underset{{\mc C}^\infty_Z}\otimes$)
$$
{\mc Hom}(p\ld({\mc
  C}^\infty_{Z,|p^*(\mcL)|^2 \otimes \Omega_Z}),{\mc
  C}^\infty_{Y,|\mcL|^2 \otimes \Omega_Y}) \simeq {\mc Hom}(p\ld({\mc
  C}^\infty_{Z,|p^*(\mcL)|^2 \otimes \Omega_{Z/Y}|}),{\mc
  C}^\infty_{Y,|\mcL|^2})
$$
on $Y$ and its global section
${\mb p}_!^{\mcL}$ corresponding to the integration map
\begin{equation}    \label{p!}
{\mb p}_!^{\mcL}: C^\infty(Z,|p^*(\mcL)|^2 \otimes \Omega_{Z/Y}) \to
C^\infty(Y,|\mcL|^2)
\end{equation}

Set
$$
\mcM := p^*(\mcL) \underset{\OO_Z}\otimes K_{Z/Y},
$$
so that
$$
|\mcM|^2 = |p^*(\mcL)|^2 \underset{{\mc C}^\infty_Z}\otimes \Omega_{Z/Y}.
$$

\begin{lemma}    \label{action2}
The sheaf ${\mc Hom}(p\ld({\mc C}^\infty_{Z,|\mcM|^2}),{\mc
  C}^\infty_{Y,|\mcL|^2})$ has a natural
$(\D_{Y,\mcL},p\ld(\D_{Z,\mcM}))$-bimo\-dule structure, which is
defined as follows: for every open $U \subset Y$, given $R \in
(p\ld(\D_{Z,\mcM}))(U) = \D_{Z,\mcM}(p^{-1}(U)), S \in
\D_{Y,\mcL}(U)$, and $\phi \in {\mc Hom}(p\ld({\mc
  C}^\infty_{Z,|\mcM|^2}),{\mc C}^\infty_{Y,|\mcL|^2})(U)$, which is a
compatible system $\{ \phi_{U'} \in
\on{Hom}(C^\infty(p^{-1}(U'),|\M|^2),C^\infty(U',|\mcL|^2)) \}$ for open
subsets $U' \subset U$,
$$
\phi \mapsto S \circ \phi \circ R,
$$
where $S \circ \phi \circ R$ stands for the compatible system $\{
S|_{U'} \circ \phi_{U'} \circ R|_{p^{-1}(U')} \, | \, U' \subset U\}$.
\end{lemma}

On the other hand, denote by $p^D_*$ the (derived) $D$-module
push-forward functor
$$
p^D_*: D^b(\D_Z) \to D^b(\D_Y).
$$
If $\mcF$ is a $\D_Z$-module, then it follows from the definition
(see \cite{Ka}, Sect. 4.6) that
\begin{equation}    \label{free}
p^D_*({\mc F}) = R^\dt p_*(\D_{Y \lr Z}
\underset{\D_Z}{\overset{L}\otimes}
{\mc F}),
\end{equation}
where
$$
\D_{Y \lr Z} := K_Z \underset{\OO_Z}\otimes \D_{Z \to Y}
\underset{p\ud(\OO_Y)}\otimes p\ud(K_Y^{-1})
$$
is a sheaf on $Z$ with a natural structure of
$(p\ud(\D_Y),\D_Z)$-bimodule.

The right action of $\D_Y$ on $K_Y$ gives rise to a canonical
isomorphism
$$
\D_Y^{\on{opp}} \simeq K_Y \underset{\OO_Y}\otimes \D_Y
\underset{\OO_Y}\otimes K_Y^{-1},
$$
which yields an identification (see \cite{Ka}, Remark 4.18)
\begin{equation}    \label{DY from Z}
\D_{Y \lr Z} \simeq p\ud(\D_Y) \underset{p\ud(\OO_Y)}\otimes
K_{Z/Y}.
\end{equation}

Let us apply $p^{D,0}_* := H^0 p^D_*$ to the left $\D_Z$-module $\D_Z
\underset{\OO_Z}\otimes \mcM^{-1}$. Since it is free,
$\overset{L}\otimes$ in formula \eqref{free} is the ordinary
$\otimes$. Note also that it carries a commuting right
$\D_{Z,\mcM}$-module structure and hence is a
$(\D_Z,\D_{Z,\mcM})$-bimodule. This implies that the sheaf $p^{D,0}_*(\D_Z
\underset{\OO_Z}\otimes \mcM^{-1})$ is a
$(\D_Y,p\ld(\D_{Z,\mcM}))$-bimodule.

Setting
$$
\D_{Y \lr Z,\mcM} := \D_{Y \lr Z}
\underset{\OO_Z}\otimes \mcM^{-1} \simeq p\ud(\D_Y)
\underset{p\ud(\OO_Y)}\otimes p^*(\mcL^{-1}),
$$
we obtain an isomorphism
$$
p^{D,0}_*(\D_Z
\underset{\OO_Z}\otimes \mcM^{-1}) \simeq p_*(\D_{Y \lr
  Z,\mcM}).
$$

Define the following $(\D_{Y,\mcL},p\ld(\D_{Z,\mcM}))$-bimodule on
$Y$:
$$
\wt\D_{Y \lr Z,\mcM} := \mcL \underset{\OO_Y}\otimes p^{D,0}_*(\D_Z
\underset{\OO_Z}\otimes \mcM^{-1}) \simeq \mcL \underset{\OO_Y}\otimes
p_*(\D_{Y \lr Z,\mcM}).
$$
It follows from the definition of the functor $p^{D,0}_*$ that the unit
$1_Z \in \D_Z$ gives rise to a global section $1_{Y \lr Z,\mcL}$ of
$\wt\D_{Y \lr Z,\mcM}$.

\begin{proposition}    \label{inj3}
There is a unique injective homomorphism of
$(\D_{Y,\mcL},p\ld(\D_{Z,\mcM}))$-bimodules
  $$
  \wt\D_{Y \lr Z,\mcM} \to {\mc Hom}(p\ld({\mc
  C}^\infty_{Z,|\mcM|^2}),{\mc C}^\infty_{Y,|\mcL|^2})
  $$
sending $1_{Y \lr Z,\mcL}$ to ${\mb p}_!^{\mcL}$.
\end{proposition}

\begin{proof}
  Let us prove the statement in the case $\mcL = \OO_Y$, so that
  $\mcM = K_{Z/Y}$.

  Recall the relative tangent sheaf $\Theta_{Z/Y}$. The sheaf
  $p_*(\Theta_{Z/Y})$ naturally acts on $p_*(K_{Z/Y})$ and on
  $p\ld({\mc C}^\infty_{Z,\Omega_{Z/Y}})$ by Lie derivatives. Hence
  $\Theta_{Z/Y}$ embeds into $p\ld(\D_{Z,K_{Z/Y}})$ as a subsheaf
  of Lie algebras, and so it acts on $p\ld(\D_{Z,K_{Z/Y}})$ by
  commutators. For every open $U \subset Y$, we have
  $$
  {\mb p}_!^{\mcL}(\xi
  \cdot f) = 0, \qquad \forall \xi \in
  p_*(\Theta_{Z/Y})(U), \quad \forall f \in
  C^\infty(p^{-1}(U),\Omega_{Z/Y}).
  $$
  It follows that
  $$
  p\ld(\D_{Z,K_{Z/Y}})^{\on{opp}} \cdot {\mb p}_!^{\mcL} \simeq 
  p\ld(\D_{Z,K_{Z/Y}})/(p\ld(\Theta_{Z/Y}) \cdot
  p\ld(\D_{Z,K_{Z/Y}})).
  $$
  Formula \eqref{DY from Z} implies that
  $$
  p\ld(\D_{Z,K_{Z/Y}})/(p\ld(\Theta_{Z/Y}) \cdot
  p\ld(\D_{Z,K_{Z/Y}})) \simeq \wt\D_{Y \lr Z,K_{Y/Z}}.
  $$
  This completes the proof of the proposition for $\mcL = \OO_Y$, so
  that $\mcM = K_{Y/Z}$. For a general line bundle $\mcL$, the
  statement of the proposition is derived in a similar way (compare
  with Propositions \ref{inj1} and \ref{inj2}).
\end{proof}

\subsection{Integral transforms}    \label{inttr}

Consider a correspondence
\begin{equation}    \label{Hecke cor2}
\begin{array}{ccccc}
& & Z & & \\
& \stackrel{p_1}\swarrow & & \stackrel{p_2}\searrow & \\
Y_1 & & & & Y_2
\end{array}
\end{equation}
Assume that $p_1$ and $p_2$ are submersions with compact
fibers. Let $\mcL_1$ and $\mcL_2$ be line bundles on $Y_1$ and $Y_2$,
respectively, such that
\begin{equation}    \label{isoma}
p_1^*(\mcL_1) \simeq p_2^*(\mcL_2) \underset{\OO_Z}\otimes
K_{Z/Y_2}.
\end{equation}
In what follows, we will fix such an isomorphism.

Consider the Hom sheaf on $Y_2$,
\begin{equation}    \label{Homsh}
{\mc Hom}(p_{2*} p_1\ud({\mc C}^\infty_{Y_1,\mcL_1}),{\mc
  C}^\infty_{Y_2,\mcL_2}).
\end{equation}
Combining Lemmas \ref{action1} and
\ref{action2}, we obtain that it is naturally a
$(\D_{Y_2,\mcL_2},p_{2*} p_1\ud(\D_{Y_1,\mcL_1}))$-bimodule.

Recall that we have a section
$$
{\mb p}_{2!}^{\mcL_2} \in {\mc Hom}(p_{2*}({\mc
  C}^\infty_{Z,|p_2^*(\mcL_2)|^2 \otimes \Omega_{Z/Y_2}}),{\mc
  C}^\infty_{Y_2,|\mcL_2|^2}) = {\mc Hom}(p_{2*}({\mc
  C}^\infty_{Z,|p_1^*(\mcL_1)|^2}),{\mc
  C}^\infty_{Y_2,|\mcL_2|^2})
$$
(here we apply the isomorphism \eqref{isoma}). We also have a section
$$
{\mb p}\ud_{1,\mc L_1} \in {\mc
  Hom}(p\ud_1({\mc C}^\infty_{Y_1,|\mcL_1|^2}),{\mc
  C}^\infty_{Z,|p_1^*(\mcL_1)|^2})
$$
which gives rise to a section
$$
p_{2*}({\mb p}\ud_{1,\mc L_1}) \in {\mc
  Hom}(p_{2*} p\ud_1({\mc C}^\infty_{Y_1,|\mcL_1|^2}),p_{2*}({\mc
  C}^\infty_{Z,|p_1^*(\mcL_1)|^2})).
$$
Denote by $H^{\mcL_1,\mcL_2}_{Z,Y_1,Y_2}$, or $H_Z$ for short, the
composition
$$
H_Z := {\mb p}_{2!}^{\mcL_2} \circ
p_{2*}({\mb p}^{-1}_{1,\mcL_1}),
$$
which is a section of the sheaf \eqref{Homsh}. This is the {\em
  integral transform} associated to the correspondence $Z$ and the
line bundles $\mcL_1$ and $\mcL_2$.

On the other hand, recall the
$(\D_{Z,p^*(\mcL_1)},p_1\ud(\D_{Y_1,\mcL_1}))$-bimodule $\D_{Z \to
  Y_1,\mcL_1}$ and the
$(\D_{Y,\mcL_2},p_{2*}(\D_{Z,\mcM_2}))$-bimodule $\wt\D_{Y \lr
  Z,\mcM_2}$, where
$$
\mcM_2 = p_2^*(\mcL_2) \underset{\OO_Z}\otimes
K_{Z/Y_2}.
$$
Formula \eqref{isoma} implies that
$$
\D_{Z,p^*(\mcL_1)} \simeq \D_{Z,\mcM_2}.
$$
Therefore we can form the tensor product
\begin{equation}    \label{DZYY}
  \D^{\mcL_1,\mcL_2}_{Y_2 \lr Z \to Y_1} := \wt\D_{Y \lr Z,\mcM_2}
  \underset{p_{2*}(\D_{Z,\mcM_2})}\otimes
  p_{2*}(\D_{Z \to Y_1,\mcL_1})
\end{equation}
which is naturally a $(\D_{Y_2,\mcL_2},p_{2*}
p_1\ud(\D_{Y_1,\mcL_1}))$-bimodule. The canonical sections of
$\wt\D_{Y \lr Z,\mcM_2}$ and $\D_{Z \to Y_1,\mcL_1}$ introduced above
give us a global section of $\D^{\mcL_1,\mcL_2}_{Y_2 \lr Z \to Y_1}$,
which we denote by $1_{Y_2 \lr Z \to Y_1}$.

Propositions \ref{inj2} and \ref{inj3} imply the following.

\begin{proposition}    \label{HH1}
  There is a unique homomorphism of $(\D_{Y_2,\mcL_2},p_{2*}
  p_1\ud(\D_{Y_1,\mcL_1}))$-bimo\-dules
  $$
  \D^{\mcL_1,\mcL_2}_{Y_2 \lr Z \to Y_1} \to {\mc
    Hom}(p_{2*} p_1\ud({\mc C}^\infty_{Y_1,|\mcL_1|^2}),{\mc
    C}^\infty_{Y_2,|\mcL_2|^2})
  $$
  sending $1_{Y_2 \lr Z \to Y_1}$ to $H_Z$.
\end{proposition}

This proposition has the following obvious corollary.

\begin{corollary}    \label{cor}
  Suppose that $P \cdot 1_{Y_2 \lr Z \to Y_1} = 0$
  for some $P \in \Gamma(Y_2,\D_{Y_2,\mcL_2})$. Then
  $$
  P \cdot H_Z = 0, \qquad \ol{P} \cdot
  H_Z = 0.
  $$
\end{corollary}

In other words, holomorphic differential equations satisfied by
the section $1_{Y_2 \lr Z \to Y_1}$ give rise to differential
equations on the integral transform $H_Z$ obtained from the
correspondence $Z$.

\medskip

Now we want to connect this to the $D$-module {\em integral transform
  functor} associated to the correspondence $Z$. Namely, we have the
functor
\begin{eqnarray}    \label{functor1}
  {\mathbf H}^\dt_Z: D^b(\D_{Y_1}) &\to & D^b(\D_{Y_2}) \\
  \label{functor}
{\mc G} &\mapsto & p^D_{2*}(p_1^*({\mc G})).
\end{eqnarray}
Let
\begin{equation}    \label{sheaf}
  {\mc H}_Z := {\mc L}_2 \underset{\OO_{Y_2}}\otimes {\mathbf
    H}^0_Z(\D_{Y_1} \underset{\OO_{Y_1}}\otimes {\mc L}^{-1}).
\end{equation}
This is a $(\D_{Y_2,\mcL_2},p_{2*}
p_1\ud(\D_{Y_1,\mcL_1}))$-bimodule. The unit $1_{Y_1}$ gives rise to a
global section of ${\mc H}_Z$, which we denote by $\psi_Z$.

Recall formula \eqref{free} for the functor $p^D_{2*}$. The next
result follows from \cite{DS}, Prop. 2.12.

\begin{proposition}[\cite{DS}]
Suppose that the map $Z \to Y_1 \times Y_2$ is a closed
embedding. Then $\D_{Y_2 \lr Z} \underset{\D_Z}{\overset{L}\otimes}
\D_{Z \to Y_1}$ is concentrated in cohomological degree $0$.
\end{proposition}

This proposition has the following corollary.

\begin{corollary}    \label{closed}
If $Z \to Y_1 \times Y_2$ is a closed embedding, then there is an
isomorphism of $(\D_{Y_2,\mcL_2},p_{2*}p_1\ud(\D_{Y_1,\mcL_1}))$-bimodules
\begin{equation}    \label{homom}
  \D^{\mcL_1,\mcL_2}_{Y_2 \lr Z \to Y_1} \simeq {\mc H}_Z
\end{equation}
under which $1_{Y_2 \lr Z \to Y_1}$ is mapped to $\psi_Z$.
\end{corollary}

Combining Corollaries \ref{cor} and \ref{closed}, we obtain the
following.

\begin{corollary}    \label{diffcorr}
Suppose that $Z \to Y_1 \times Y_2$ is a closed
embedding. If $P
\cdot \psi_Z = 0$ for some $P \in \Gamma(Y_2,\D_{Y_2,\mcL_2})$, then
  $$
  P \cdot H_Z = 0, \qquad \ol{P} \cdot
  H_Z = 0.
  $$
\end{corollary}

In our proof of Theorem \ref{mainthm} in the next section we will need
the following variant of Corollary \ref{diffcorr}. Let $\mcL'_2$ be
another line bundle on $Y_2$ and
$$
\D^{\mcL'_2}_{Y_2,\mcL_2} := \mcL'_2 \underset{\OO_{Y_2}}\otimes
\mcL_2^{-1} \underset{\OO_{Y_2}}\otimes \D_{Y_2,\mcL_2},
$$
the sheaf of differential operators acting from $\mcL_2$ to $\mcL'_2$
on $Y_2$.

\begin{corollary}    \label{diffcorr1}
Suppose that $Z \to Y_1 \times Y_2$ is a closed embedding. If $P \cdot
\psi_Z = 0$ for some $P \in \Gamma(Y_2,\D^{\mcL'_2}_{Y_2,\mcL_2})$, then
  $$
  P \cdot H_Z = 0, \qquad \ol{P} \cdot
  H_Z = 0.
  $$
\end{corollary}

\section{The case of $PGL_n$}    \label{PGLn}

In this section we consider the case of $G=PGL_n$ (so that $\LG=SL_n$)
and the Hecke correspondence $\ol{Z}(\omega_1)$ associated to $\la =
\omega_1$, the first fundamental coweight of $PGL_n$. There are two
advantages in this case that we will exploit: (1) $\ol{Z}(\omega_1) =
Z(\omega_1)$, i.e. the fibers of $q_2 \times q_3$ are isomorphic to a
closed $G[[z]]$-orbit $\Gr_{\omega_1}$, which is smooth and compact
(in fact, $\Gr_{\omega_1} \simeq \mP^{n-1}$); and (2) the special
section (corresponding to the oper Borel reduction) of the universal
oper bundle $\V_{\omega_1}$ satisfies an $n$th order differential
equation (see Lemma \ref{Dga}). Hence we will be able to apply
Corollary \ref{diffcorr1} in the case of the Hecke correspondence
$Z(\omega_1)$ to derive the differential equations \eqref{eqHom1} on
the Hecke operator $\wh{H}_{\omega_1}$ and thus prove Theorem
\ref{mainthm}.

Note that we also have the multiplicity one property (see
Proposition \ref{mone} and Theorem \ref{RS}). We will use the
notation of Section \ref{notation}.

\subsection{Hecke functor}

We follow the definition of the Hecke functor ${\mb H}^\dt_{\omega_1}$
given in \cite{BD} (where it is denoted by $T^\dt_{\omega_1}$), taking
into account the fact that in this case the fibers of the morphism
$q_2 \times q_3$ are isomorphic to $\ol\Gr_{\omega_1} \simeq {\mathbb
  P}^{n-1}$ and hence are smooth. The following definition is taken
from \cite{BD}, Sect. 5.2.4.

\begin{definition}[\cite{BD}]
  {\em The Hecke functor
    $$
    {\mb H}^\dt_{\omega_1}: D^b(\D_{\Bun_{PGL_n}}) \to
      D^b(\D_{\Bun_{PGL_n} \times X})
    $$
    is defined as follows. For a left $D$-module $M$ on
    $\Bun_{PGL_n}$,
\begin{equation}    \label{Tla}
{\mb H}^\dt_{\omega_1}(M) := (q_2 \times q_3)^D_*(q_1^*(M)),
\end{equation}
where $(q_2 \times q_3)^D_*$ denotes the derived direct image
functor for $D$-modules.}
\end{definition}

Denote by ${\mb H}^i_{\omega_1}$ the corresponding $i$th cohomology
functor. We now recall a theorem of Beilinson and Drinfeld \cite{BD}
describing the action of ${\mb H}^i_{\omega_1}$ on a specific
$D$-module on $\Bun_{PGL_n}$.

If $n$ is even, then to define this $D$-module, we need to pick a
square root $K_X^{1/2}$ of the canonical line bundle on $X$ (as in
Section \ref{intPGLn}). To it, one associates a specific square root
\begin{equation}    \label{sqroot}
  \mcL = K_{\Bun}^{1/2}
\end{equation}
of the canonical line bundle on $\Bun_{PGL_n}$ following
\cite{BD}, Sect. 4 (see also \cite{LS}). If $n$ is odd, the construction of
$\mcL$ does not require any choices.

Let $\D_{\Bun_{PGL_n}}$ be the sheaf of differential operators on
$\Bun_{PGL_n}$. Then
$$
\D_{\Bun_{PGL_n}} \otimes \mcL^{-1}
$$
is a left $D$-module on $\Bun_{PGL_n}$ equipped with the commuting
right action of the algebra
$$
D_{PGL_n} = D^\beta_{PGL_n} = \Gamma(\Bun^\beta_{PGL_n}, \mcL \otimes
\D_{\Bun_{PGL_n}} \otimes \mcL^{-1})
$$
Recall from Definition \ref{beta} that the right hand
  side does not depend on $\beta$. Moreover, we have
$$
D_{PGL_n} \simeq \on{Fun} \on{Op}^\ga_{SL_n}(X)
$$
according to the result of \cite{BD} which is quoted in Theorem
\ref{DGOp} above. Here $\ga$ denotes the isomorphism
  class of $K_X^{1/2}$, as in Theorem \ref{DGOp}.

\subsection{Hecke eigensheaf property}

Recall the left $D$-module $\V_{\omega_1,X}^{\on{univ}}$ on $X$
obtained from the universal oper bundle. It is equipped with a
commuting action of the above algebra $D_{PGL_n}$. Furthermore, by
definition of $SL_n$-opers, we have an embedding
$$
\kappa_{\omega_1}^{\on{univ}}: K_X^{(n-1)/2} \hookrightarrow
\V_{\omega_1,X}^{\on{univ}}
$$
and hence a section
$$
s_{\omega_1}^{\on{univ}} \in \Gamma(X,K_X^{-(n-1)/2} \otimes
\V_{\omega_1,X}^{\on{univ}}).
$$

Recall the $n$th order differential operator \eqref{sigman}. Lemmas
\ref{Dga} and \ref{nsigma} imply the following differential equation
on $s_{\omega_1}^{\on{univ}}$ (this is an analogue of equation
\eqref{seq1} in the case of $GL_1$):
\begin{equation}    \label{sigmas}
\sigma \cdot s_{\omega_1}^{\on{univ}} = 0.
\end{equation}
We will now use this equation to derive the system \eqref{eqHom1}.

Consider the isomorphism \eqref{cana}:
\begin{equation}    \label{cana1}
a : q_1^* (\mcL^{1/2}) \overset{\sim}\longrightarrow q_2^*
(\mcL)\otimes K_2 \otimes
q_3^* (K_X ^{-(n-1)/2})
\end{equation}
It gives rise to a section $\psi_{Z(\omega_1)}$ of $(\mcL
\boxtimes K_X^{-(n-1)/2}) \otimes {\mb
  H}^0_{\omega_1}(\D_{\Bun_{PGL_n}} \otimes \mcL^{-1})$.

The first part of the following theorem is Theorem 5.2.9 of
\cite{BD}. The second part follows from Theorems 5.4.11, 5.4.12, and
Proposition 8.1.5 of \cite{BD}.

\begin{theorem}[\cite{BD}]    \label{Heceigen}
  ${\mb H}^i_{\omega_1}(\D_{\Bun_{PGL_n}}
  \otimes \mcL^{-1}) = 0$ for $i \neq 0$, and
  $${\mb H}^0_{\omega_1}(\D_{\Bun_{PGL_n}}
  \otimes \mcL^{-1}) \simeq (\D_{\Bun_{PGL_n}}
  \otimes \mcL^{-1}) \underset{D_{PGL_n}}\boxtimes
  \V_{\omega_1,X}^{\on{univ}}
  $$
  as left $\D_{\Bun_{PGL_n}} \boxtimes \D_X$-modules equipped with a
  commuting action of $D_{PGL_n}$.

  Moreover, the section
  $\psi_{Z(\omega_1)}$ of
  $$(\mcL \boxtimes K_X^{-(n-1)/2}) \otimes {\mb
    H}^0_{\omega_1}(\D_{\Bun_{PGL_n}} \otimes
  \mcL^{-1}) \simeq (\mcL \otimes \D_{\Bun_{PGL_n}}
  \otimes \mcL^{-1}) \underset{D_{PGL_n}}\boxtimes (K_X^{-(n-1)/2} \otimes
  \V_{\omega_1,X}^{\on{univ}})
  $$
  coincides with $1 \boxtimes s_{\omega_1}^{\on{univ}}$.
\end{theorem}

Theorem \ref{Heceigen} and equation \eqref{sigmas} immediately imply:

\begin{corollary}    \label{psiZ}
  The section $\psi_{Z(\omega_1)}$ satisfies
    \begin{equation}    \label{sigmapsi}
    \sigma \cdot \psi_{Z(\omega_1)} = 0.
    \end{equation}
\end{corollary}

Now we derive Theorem \ref{mainthm} from Corollary \ref{psiZ}
using Corollary \ref{diffcorr1} in the case of the Hecke correspondence
$Z(\omega_1)$.

\subsection{Proof of Theorem \ref{mainthm}}    \label{derive}
We are going to derive Theorem \ref{mainthm} from Corollary
\ref{diffcorr1} with $Z, Y_1, Y_2, \mcL_1, \mcL_2$, and $\mcL'_2$
defined below. Initially, we would like to take $Y_1 = \Bunn_{PGL_n}$,
$Y_2 = \Bunn_{PGL_n} \times X$, and $Z = Z(\omega_1)$ (see Section
\ref{notation}). However, in order for the morphism $Z \to Y_1 \times
Y_2$ to be a closed embedding and the morphism $Z \to Y_2$ to be
proper, we need to restrict $Z(\omega_1)$ to open dense subsets on
both sides, as we now explain.

Recall from formula \eqref{UG} that we have an open dense subvariety
of $\Bunn_{PGL_n}$,
\begin{equation}    \label{UG1}
U_{PGL_n}(\omega_1) = \{ \mcF \in \Bunn_{PGL_n} \, | \,
q_2(q_1^{-1}(\mcF)) \subset \Bunn_{PGL_n} \},
\end{equation}
which is dense by our assumption. It follows from
  \cite{NR}, Lemma 5.9,\footnote{We thank Tony Pantev for this
  reference and a helpful discussion. Note that the correspondence
  considered in \cite{NR} differs from the Hecke correspondence
  $Z(\omega_1)$ in that one of the two bundles is dualized; but since
  the dual of a stable bundle is stable, we can use Lemma 5.9 of
  \cite{NR} in our setting.} that there exist open dense subsets
  $$
  Y_1 \subset U_{PGL_n}(\omega_1), \qquad Y_2 \subset
  \Bunn_{PGL_n} \times X
  $$
such that the restriction $Z$ of $Z(\omega_1)$ to $Y_1 \times Y_2$
  is a closed embedding and the corresponding map $q_2 \times q_3: Z
  \to Y_2$ is proper. Denote by $p_1$ and $p_2$ the maps $q_1$ and
  $q_2 \times q_3$ restricted to $Y_1 \times Y_2$,
  respectively. Finally, set $\mcL_1 = K^{1/2}_{\Bun}$ and $\mcL_2 =
  K^{1/2}_{\Bun} \boxtimes K_X^{-(n-1)/2}$. Thus, we have $|\mcL_1|^2
  = \Omega^{1/2}_{\Bun}$ and $|\mcL_2|^2 = \Omega^{1/2}_{\Bun} \boxtimes
  \Omega_X^{-(n-1)/2}$. Then we have the isomorphism \eqref{isoma}
  which follows from the isomorphism \eqref{cana1} (see also formula
  \eqref{sqroot}).

Now we are in the setting of Corollary \ref{diffcorr1}, with $\mcL'_2
:= K^{1/2}_{\Bun} \boxtimes K_X^{(n+1)/2}$. Let $H_Z$ be the
corresponding integral transform operator. We then obtain from
Corollary \ref{diffcorr1} and equation \eqref{sigmapsi} that $H_Z$
satisfies the system of differential equations
\begin{equation}    \label{eqHom2}
    \sigma \cdot H_Z = 0, \qquad \ol\sigma \cdot H_Z = 0.
\end{equation}

Next, we relate $H_Z$ to our Hecke operator $\wh{H}_{\omega_1}$.
Recall from Section \ref{notation} that the operator
$\wh{H}_{\omega_1}$ is defined as the integral transform via the
correspondence $Z(\omega_1)$ from the space $V_{PGL_n}(\omega_1)$ of
smooth compactly supported sections of $|\mcL_1|^2 =
\Omega^{1/2}_{\Bun}$ on $U_{PGL_n}(\omega_1)$ to the space $V_{PGL_n}
\otimes \Gamma(X,\Omega_X^{-(n-1)/2})$ of smooth sections of
$|\mcL'_2|^2 = \Omega^{1/2}_{\Bun} \boxtimes \Omega_X^{(n+1)/2}$ on
$\Bunn_{PGL_n} \times X$ (in fact, the image consists of compactly
supported sections). Let $\wt{H}_{\omega_1}$ be the restriction of
$\wh{H}_{\omega_1}$ to the space of smooth sections of $|\mcL_1|^2$
that are compactly supported on $Y_1 \subset U_{PGL_n}(\omega_1)$,
followed by the restriction to $Y_2$ of the resulting section of
$|\mcL'_2|^2$ on $\Bunn_G \times X$. Thus, $\wt{H}_{\omega_1}$ acts
from the space of smooth compactly supported sections of $|\mcL_1|^2$
on $Y_1$ to the space of smooth sections of $|\mcL'_2|^2$ on $Y_2$.

It follows from the above definition that $\wt{H}_{\omega_1}$ is the
restriction of $H_Z$ to smooth compactly supported sections of
$|\mcL_1|^2$ on $Y_1$. Therefore $\wt{H}_{\omega_1}$ also
satisfies the system \eqref{eqHom2}. Since $Y_1$ is dense in
$U_{PGL_n}(\omega_1)$ and $Y_2$ is dense in $\Bunn_{PGL_n}
\times X$, this implies that $\wh{H}_{\omega_1}$ also satisfies this
system of equations. Thus, we obtain the system \eqref{eqHom1}. This
completes the proof of Theorem \ref{mainthm}.\qed

\begin{remark}
  It is possible to write an explicit formula for the differential
  operator $\sigma$ similar to formula \eqref{sigma-GL1} in the case
  of $GL_1$. For example, let $G=PGL_2$, so $\LG=SL_2$. Then the space
  $\on{Op}_{SL_2}^\ga(X)$ is an affine space over the vector space
  $H^0(X,K_X^2)$. Let us pick a point $\chi_0 \in
  \on{Op}_{SL_2}^\ga(X)$ and use it to identify
  $\on{Op}_{SL_2}^\ga(X)$ with $H^0(X,K_X^2)$. Let $\{ \varphi_i,
  i=1,\ldots,3g-3 \}$ be a basis of $H^0(X,K_X^2)$ and $\{ F_i,
  i=1,\ldots,3g-3 \}$ the set of generators of the polynomial algebra
  $\on{Fun} \on{Op}_{SL_2}^\ga(X)$ dual to this basis,
  i.e.
  $$
  F_i(\chi_0+\varphi_j) = \delta_{ij}.
$$
  Let $\{ D_i, i=1,\ldots,3g-3 \}$ be the global holomorphic
  differential operators on $\Bun_{PGL_2}$ corresponding to the $F_i$
  under the isomorphism $\on{Fun} \on{Op}_{SL_2}^\ga(X) \simeq
  D_{PGL_2}$.

  By Lemma \ref{Dga}, we have an isomorphism $\on{Op}^\ga_{SL_2}(X)
  \simeq D^\ga_2(X)$, where $D^\ga_2(X)$ is the space of projective
  connections on $X$ (corresponding to our choice of $K_X^{1/2}$).
  It sends
  $$
  \chi \in \on{Op}^\ga_{SL_2}(X) \quad \mapsto \quad P_\chi \in
  D^\ga_2(X).
  $$
Then we can write
\begin{equation}    \label{sigmaOp}
\sigma = P_{\chi_0} + \sum_{i=1}^{3g-3} D_i  \otimes \varphi_i
\; : \; K_X^{-1/2} \to D_{PGL_2} \otimes K_X^{3/2}
\end{equation}
(it is clear that this operator does not depend on the choice of
$\chi_0$).

Note that locally on $X$, after choosing a local coordinate $z$, we
can write the second-order differential operator $P_{\chi_0} \in
D^\ga_2(X)$ in the form
$$
P_{\chi_0} = \pa_z^2 + v_0(z).
$$

This gives a more concrete realization of the equations
\eqref{eqHom1}.\qed
\end{remark}

\begin{remark}    \label{chi1}
  For $\chi \in \on{Op}^\ga_{PGL_n}(X) = \on{Spec} D_{PGL_n}$, let
  $\C_\chi$ be the corresponding one-dimensional
  $D_{PGL_n}$-module. In \cite{BD}, Sect. 5.1.1, Beilinson and
  Drinfeld defined the following left $D$-module on $\Bun_{PGL_n}$:
$$
  \Delta^0_\chi := (\D_{\Bun_{PGL_n}} \otimes \mcL^{-1})
  \underset{D_{PGL_n}}\otimes \C_\chi
$$
  (this is the $D$-module
  from Theorem \ref{RS}). They derived from Theorem \ref{Heceigen}
  that ${\mb H}^i_{\omega_1}(\Delta^0_\chi) = 0$ for $i \neq 0$ and
$$
  {\mb H}^0_{\omega_1}(\Delta^0_\chi) \simeq \Delta^0_\chi
  \boxtimes (\V_{\omega_1},\nabla_\chi)
  $$
  (see \cite{BD}, Theorem 5.2.6). This means that $\Delta^0_\chi$ is
  a {\em Hecke eigensheaf} with respect to the flat $SL_n$-bundle
  corresponding to $\chi$ \cite{BD}. In this sense, $\D_{PGL_n}
  \otimes \mcL^{-1}$ is a universal Hecke eigesheaf parametrized by
  the component $\on{Op}^\ga_{SL_n}(X)$ of the space of $SL_n$-opers
  on $X$.

Recall the equivalence of categories in the abelian case obtained by
restriction of the $D$-module version of the Fourier--Mukai transform
discussed to Remark \ref{chii}. It has a non-abelian analogue (see
e.g. \cite{FT}). In the case of $G=PGL_n$, on one side we have the
category of coherent sheaves on $\on{Op}^\ga_{PGL_n}(X)$. On the other
side, we have the category of $D$-modules ${\mc K}$ on $\Bun_{PGL_n}$
with finite global presentation of the form
      $$
      (\D_{\Bun{PGL_n}} \otimes \mcL^{-1})^{\oplus m} \to (\D_{\Bun_{PGL_n}}
      \otimes \mcL^{-1})^{\oplus r} \to
             {\mc K}
      \to 0.
      $$
The equivalence $E$ takes an object ${\mc F}$ of the former category to
    $$
    E({\mc F}) := (\D_{\Bun_{PGL_n}} \otimes \mcL^{-1})
    \underset{D_{PGL_n}}\otimes F, \qquad F :=
    \Gamma(\on{Op}^\ga_{PGL_n}(X),{\mc F}).
    $$
    In particular, $E(\OO_{\on{Op}^\ga_{PGL_n}(X)}) = \D_{\Bun_{PGL_n}}
    \otimes \mcL^{-1}$ and $E(\OO_\chi) = \Delta^0_\chi$.  \qed
\end{remark}

\section{General case}    \label{gencase}

In this section we formulate the analogues of Theorem \ref{mainthm}
and Corollary \ref{eigH1} describing the eigenvalues of the Hecke
operators in the case of an arbitrary simple Lie group $G$ and outline
their proof following the argument of the previous section.

The case of $G=PGL_n$ and $\la = \omega_1$, which we considered in the
previous section, differs from the case of a general group $G$ and
dominant integral coweight $\lambda \in P^\vee_+$ in two ways. First,
in the case of $G=PGL_n$ and $\la = \omega_1$ we have $\ol{Z}(\la) =
Z(\la)$, i.e. the fibers $(\ol{q}_2 \times \ol{q}_3)^{-1}({\mc P},x)$
of the Hecke correspondence $\ol{Z}(\omega_1)$ are isomorphic to the
$PGL_n[[z]]$-orbit $\Gr_{\omega_1}$ in the affine Grassmannian of
$PGL_n$ which is smooth and compact. But for general $G$ and $\la \in
P^\vee_+$ these fibers are isomorphic to the closure of the
$G[[z]]$-orbit $\Gr_\la$ and are singular. Second, in the case of
$G=PGL_n$ and $\la = \omega_1$, the canonical section of the universal
oper bundle satisfies a scalar differential equation of Lemma
\ref{Dga}, but for general $G$ and $\lambda$ this is not the
case. Hence for general $G$ and $\lambda$ our construction needs some
modifications, which we discuss in this section.

Namely, we formulate an analogue (Conjecture \ref{eigHG}) of Corollary
\ref{eigH1} describing the eigenvalues of Hecke operators in terms of
real $\LG$-opers and an analogue (Conjecture \ref{eqHG}) of Theorem
\ref{mainthm} describing differential equations satisfied by the Hecke
operators. Under the irreducibility assumption of Corollary
\ref{monG}, Conjecture \ref{eigHG} follows from Conjecture \ref{eqHG}
in the same way as in the case of $PGL_n$.\footnote{In
  \cite{EFK4}, Proposition 3.39, we derive the statement of
  Conjecture \ref{eigHG} from Conjecture \ref{eqHG} without the
  irreducibility assumption of Corollary \ref{monG}.}

We expect that Conjecture \ref{eqHG}
can be derived from the Hecke eigensheaf property (established in
\cite{BD} and recalled in Section \ref{HeckeGen}) by an argument
analogous to the one we used in the proof of Theorem \ref{mainthm} in
Section \ref{derive} in the case of $PGL_n$. However, because the
morphism $\ol{q}_2 \times \ol{q}_3$ is not smooth in the general case
(in the sense of algebraic geometry), deriving Conjecture \ref{eqHG}
from this result requires additional care. We leave the details to a
follow-up paper.

\subsection{Eigenvalues of the Hecke operators}

Recall from the discussion before Theorem \ref{DGOp} that the space
$\on{Op}_{\LG}(X)$ of $\LG$-opers on $X$ has a canonical component
$\on{Op}^\ga_{\LG}(X)$ isomorphic to the affine space
$\on{Op}_{\LG_{\on{ad}}}(X)$. If the set $\{ \langle \la,\rho
\rangle, \la \in P^\vee_+ \}$ contains half-integers, then to specify
$\on{Op}^\ga_{\LG}(X)$ we need to choose a square root $K_X^{1/2}$ of
the canonical line bundle on $X$; then $\gamma$ denotes the
isomorphism class of this $K_X^{1/2}$ which we also use to construct a
square root $\mcL = K^{1/2}_{\Bun}$ of the canonical line bundle on
$\Bun_G$ (see the discussion before Theorem \ref{h}). Otherwise, the
component $\on{Op}^\ga_{\LG}(X)$ is well-defined without any choices.

For $\la \in P^\vee_+$, let $V_\la$ be the corresponding irreducible
finite-dimensional representation of $\LG$. Given $\chi \in
\on{Op}^\ga_{\LG}(X)$, we obtain a flat holomorphic vector bundle
$(\V_{\chi,\la},\nabla_{\chi,\la})$ on $X$. According to
  \cite{BD}, \S 3, the vector bundles $\V_{\chi,\la}$ are isomorphic
  to each other for all $\chi \in \on{Op}^\ga_{\LG}(X)$. Hence we will
  use the notation $\V_{\la}$.

The oper Borel reduction gives rise to an embedding
\begin{equation}    \label{kala}
  \kappa_\la: K_X^{\langle \la,\rho \rangle} \hookrightarrow
  \V_\la
\end{equation}
  and hence
  $$
  \wt\kappa_\la: \OO_X \hookrightarrow K_X^{-\langle \la,\rho \rangle} \otimes
  \V_\la.
  $$
  Let
  $$
  s_\la := \wt\kappa_\la(1) \in \Gamma(X,K_X^{-\langle \la,\rho \rangle} \otimes
  \V_\la).
  $$

Now suppose that $\chi \in \on{Op}^\ga_{\LG}(X)_{\R}$. Then we have an
isomorphism of $C^\infty$ flat bundles 
$$
  (\V_\la,\nabla_{\chi,\la}) \simeq (\ol{\V}_\la,\ol\nabla_{\chi,\la})
$$
and hence a pairing
  $$
  h_{\chi,\la}(\cdot,\cdot): (\V_\la,\nabla_{\chi,\la}) \otimes
  (\ol{\V}_{-w_0(\la)},\ol\nabla_{\chi,-w_0(\la)}) \to ({\mc
    C}^\infty_X,d)
  $$
  as $V^*_\la \simeq V_{-w_0(\la)}$. Since $\langle
  -w_0(\la),\rho \rangle = \langle \la,\rho \rangle$, we have
  $$
  \ol{s_{-w_0(\la)}} \in \Gamma(X,\ol{K}_X^{-\langle
    \la,\rho \rangle} \otimes \ol{\V}_{-w_0(\la)}.
  $$

  Recall that $\Bun_G$ has connected components $\Bun_G^\beta, \beta
  \in \pi_1(G)$, and we have a direct sum decomposition
\begin{equation}    \label{dirsum1}
\mcH_G = \bigoplus_{\beta \in \pi_1(G)} \mcH_G^\beta.
\end{equation}
According to Conjecture \ref{SpecA}, for each $\chi \in
\on{Op}^\ga_{\LG}(X)_{\R}$ we have a non-zero eigenspace of ${\mc
  A}_G$ in $\mcH_G^\beta$ for all $\beta \in \pi_1(G)$ (which is
one-dimensional by Proposition \ref{mone} and Theorem
\ref{RS}). We expect that the Hecke operator $H_\la$ preserves the
direct sum of these subspaces (see Section \ref{intPGLn} in the case
of $PGL_n$). Moreover, one can find out precisely how $H_\la$ permutes
different $\mcH_G^\beta$ by analyzing the action of the center
$Z(\LG)$ on $V_\la$ (which is naturally identified with the group of
characters of $\pi_1(G)$). As in the case of $PGL_n$, this implies
that the eigenvalues $\{ \Phi_\la(\chi) \}$ of $H_\la$ corresponding
to $\chi$ form a torsor over the group $\mu_{\alpha(G,\la)}$ of roots
of unity of some order $\alpha(G,\la)$; for example,
$\alpha(PGL_n,\omega_1)=n$ (this is so even for $G=SO_{4n}/\{ \pm I
\}$, when $\pi_1(G) \simeq \Z_2 \times \Z_2$). This is why we describe
these eigenvalues up to a scalar.

The following conjecture is an analogue of Corollary \ref{eigH1} for a
general group $G$.

\begin{conjecture}    \label{eigHG}
    For $\chi \in \on{Op}^\ga_{\LG}(X)_{\R}$, the section
    $\Phi_\la(\chi) \in \Gamma(X,\Omega_X^{-\langle \la,\rho \rangle})$
    is equal to
    \begin{equation}    \label{Phila}
    \Phi_\la(\chi) = h_{\chi,\la}(s_\la,\ol{s_{-w_0(\la)}})
    \end{equation}
    up to a scalar.
\end{conjecture}

Let us choose a system of embeddings of representations
    of $\LG$,
      \begin{equation}    \label{imath}
      \imath_{\la,\mu}: V_{\la+\mu} \hookrightarrow V_\la \otimes
      V_\mu, \qquad \la,\mu \in P^\vee_+
      \end{equation}
      satisfying $\imath_{\la,\mu} \circ \imath_{\la+\mu,\nu} =
      \imath_{\mu,\nu} \circ \imath_{\la,\mu+\nu}$. These give rise to
      a system of linear maps
  $$
  \wt\imath_{\la,\mu}: \Gamma(X,K_X^{-\langle \la+\mu,\rho \rangle} \otimes
  \V_{\la+\mu}) \to \Gamma(X,K_X^{-\langle \la,\rho \rangle} \otimes
  \V_\la) \otimes \Gamma(X,K_X^{-\langle \mu,\rho \rangle} \otimes
  \V_\mu)
  $$

  \begin{lemma}    \label{plucker}
  The sections $s_\la, \la \in P^\vee_+$, satisfy
  \begin{equation}
    \wt\imath_{\la,\mu}(s_{\la+\mu}) = s_\la \otimes s_\mu, \qquad \forall
    \la,\mu \in P^\vee_+.
  \end{equation}
  \end{lemma}

\begin{remark}
By Lemma \ref{plucker}, the sections
  $\Phi_\la(\chi)$ given by formula \eqref{Phila} satisfy
  \begin{equation}    \label{mult}
  \Phi_{\la+\mu}(\chi) = \Phi_\la(\chi) \Phi_\mu(\chi), \qquad \la,\mu
  \in P^\vee_+,
  \end{equation}
  provided that the pairings $h_{\chi,\la}, \la \in P^\vee_+$, are
  normalized so that the restriction of $h_{\chi,\la} \otimes
  h_{\chi,\mu}$ to the image of the embedding obtained from
  $\imath_{\la,\mu}$ (see formula \eqref{imath}) is equal to
  $h_{\chi,\la+\mu}$. Formula \eqref{mult} agrees with the relations
  $H_{\la+\mu} = H_\la \cdot H_\mu$ satisfied by the Hecke operators
  according to Proposition \ref{algiso}.
  
As in the case of $G=PGL_n$ (see Section \ref{intPGLn}), we expect
that this normalization condition fixes the pairings $h_{\chi,\la},
\la \in P^\vee_+$ (and hence the eigenvalues $\Phi_\la(\chi)$ of the
Hecke operators) up to a root of unity of order $\alpha(G,\la)$.
Conjecture \ref{eigHG} can then be refined to a statement that two
torsors over the corresponding group $\mu_{\alpha(G,\la)}$ of roots of
unity are equal to each other, similarly to Conjecture \ref{eigH} for
$G=PGL_n$.\qed
\end{remark}

\subsection{Analogue of the system of differential equations}

To prove Conjecture \ref{eigHG}, we need an analogue of the system of
differential equations which appear in Corollary \ref{unique} and
Theorem \ref{mainthm} in the case of $PGL_n$ and $\la=\omega_1$. In
that case, we used the possibility to interpret opers in terms of
scalar differential operators of order $n$ (see Lemma \ref{Dga}). Such
an interpretation is possible if $\la$ is such that $V_\la$ remains
irreducible under a principal $\sw_2$ subalgebra of $\lg$. It is known
that this happens for $\la=\omega_1$ and $\omega_\ell$ if $\g$ is of
type $A_\ell$, and for $\la=\omega_1$ if $\g$ is of types $B_\ell$,
$C_\ell$ and $G_2$.\footnote{This case is discussed in
  detail in \cite{EFK4}, Sect. 3.6.} For $A_\ell, B_\ell$, and $C_\ell$, the
corresponding scalar differential operators were described in
\cite{DrS,BD:opers}.  In the general case, we replace the equations
from Corollary \ref{unique} and Theorem \ref{mainthm} with a statement
(Lemma \ref{mon}) about the (twisted) $D$-module on $X$ obtained by
applying {\em all} possible (twisted) differential operators to the
canonical section $s_\la$.

Namely, let
$$
\D_{X,-\langle \la,\rho \rangle} := K_X^{-\langle \la,\rho \rangle}
\underset{\OO_X}\otimes \D_X \underset{\OO_X}\otimes K_X^{\langle
  \la,\rho \rangle}
$$
be the sheaf of differential operators acting on the line bundle
$K_X^{-\langle \la,\rho \rangle}$ on $X$. Set
$$
\mcV^K_\la := K_X^{-\langle \la,\rho \rangle} \underset{\OO_X}\otimes
\mcV_\la.
$$
The oper connection $\nabla_{\chi,\la}$ defines a $\D_X$-module
structure on $\V_\la$, and therefore a $\D_{X,-\langle \la,\rho
  \rangle}$-module structure on $\V^K_\la$. We will denote this
$\D_{X,-\langle \la,\rho \rangle}$-module by $\V^K_{\chi,\la}$. Recall
that we have a canonical section $s_\la \in \Gamma(X,\mcV^K_\la)$.

\begin{lemma}    \label{mon}
If the monodromy representation of the flat vector bundle
$(\mcV_\la,\nabla_{\chi,\la})$ is irreducible,\footnote{Denote by
$M_\chi$ the subgroup of $\LG$ which is the Zariski
  closure of the monodromy of $\chi \in \on{Op}^\ga_{\LG}(X)$. In
  \cite{EFK4}, Corollary 3.32, we prove that the subset of those
  $\chi$ for which $M_\chi=\LG$ (and so the flat vector bundles
  $(\mcV_\la,\nabla_{\chi,\la})$ are irreducible for all $\la \in
  P^\vee_+$) is dense in $\on{Op}^\ga_{\LG}(X)$. In general, according
  to \cite{EFK4}, Theorem 3.36, $M_\chi$ is a simple subgroup of $\LG$
  whose Lie algebra contains a principal $\sw_2$ subalgebra of $\lg$,
  and in this case $\chi$ is induced from an $M_\chi$-oper $\eta$ (if
  $\chi$ is a real oper, then so is $\eta$). For such $\LG$-opers,
  $(\mcV_\la,\nabla_{\chi,\la})$ decomposes into a direct sum of
  irreducible flat vector bundles according to the decomposition of
  $V_\la$ into irreducible representations of $M_\chi$.} then
\begin{equation}    \label{equal}
\D_{X,-\langle \la,\rho \rangle} \cdot s_\la = \mcV^K_{\chi,\la}.
\end{equation}
\end{lemma}

\begin{proof}
Note that $\D_{X,-\langle \la,\rho \rangle} \cdot s_\la$ is a non-zero
$\D_{X,-\langle \la,\rho \rangle}$-submodule of $\V^K_{\chi,\la}$. If
the monodromy representation of the flat vector bundle
$(\mcV_\la,\nabla_{\chi,\la})$ is irreducible, then
$(\mcV_\la,\nabla_{\chi,\la})$ is an irreducible $\D_X$-module, and so
$\V^K_{\chi,\la}$ is an irreducible $\D_{X,-\langle \la,\rho
  \rangle}$-module. Hence \eqref{equal} follows.
\end{proof}

Let $I_{\la,\chi}$ be the left annihilating ideal of $s_\la$ in the
sheaf $\D_{X,-\langle \la,\rho \rangle}$. Thus, we have an exact
sequence of left $\D_{X,-\langle \la,\rho \rangle}$-modules
$$
0 \to I_{\la,\chi} \to \D_{X,-\langle \la,\rho \rangle} \to
\mcV^K_{\chi,\la} \to 0
$$
Lemma \ref{mon} has the following immediate corollary.\footnote{In
\cite{EFK4}, Proposition 3.39, we prove the statement of Corollary
\ref{monG} in general (without assuming that the monodromy
representation of $(\mcV_\la,\nabla_{\chi,\la})$ is irreducible). In
\cite{EFK4}, Sect. 3.9, we use this result to describe an analogue of
the Langlands functoriality principle in the analytic Langlands
correspondence for curves over $\C$.}

\begin{corollary}    \label{monG}
  Suppose that the monodromy representation of the flat vector bundle
  $(\mcV_\la,\nabla_{\chi,\la})$, where $\chi \in \on{Op}^\ga_{\LG}(X)_{\R}$,
  is irreducible. Then $h_{\chi,\la}(s_\la, \ol{s_{-w_0(\la)}})$ is a
  unique, up to a scalar, non-zero section $\Psi_\la(\chi)$ of
  $\Omega_X^{-\langle \la,\rho \rangle}$ annihilated by the ideals
  $I_{\la,\chi}$ and $\ol{I_{-w_0(\la),\chi}}$.
\end{corollary}

\begin{remark}    \label{diffops}
 For $G=PGL_n$, $\la = \omega_1$, let
$$
I'_{\omega_1,\chi} := K_X^n \otimes I_{\omega_1,\chi},
$$
which is a left submodule of the $({\mc
  D}_{X,(n+1)/2},{\mc D}_{X,(-n+1)/2})$-bimodule of differential
operators acting from $K_X^{(-n+1)/2}$ to $K_X^{(n+1)/2}$. It has
the property that it is generated by a globally defined
$n$th order differential operator $P_\chi$ on $X$ associated to $\chi$
by Lemma \ref{Dga}, i.e.
$$
I'_{\omega_1,\chi} = {\mc D}_{X,(n+1)/2} \cdot P_\chi.
$$
Therefore in this case a section
annihilated by the ideal $I_{\la,\chi}$ is the same as a section
satisfying the $n$th order differential equation \eqref{diffeqs}.

For a general simple Lie group $G$, the ideal $I_{\la,\chi}$, or its
twist such as $I'_{\la,\chi}$, doesn't have such a
generator. But as the above corollary shows, this is not
necessary. What matters is the cyclicity of the $D$-module
$\mcV^K_{\chi,\la}$ (formula \eqref{equal}).

We note that it is the cyclicity of two types of twisted $D$-modules:
$\mcV^K_{\chi,\la}$ on $X$ and $\Delta_\chi$ on $\Bun_G$ (see formula
\eqref{Deltachi}), that enables us to link the geometric Langlands
correspondence and the analytic one (see Remark
\ref{Grothendieck}).\qed
\end{remark}

\begin{remark} Note that a non-zero section $\Psi_\la(\chi)$ from Corollary
  \ref{monG} satisfies
  $$
  (\D_{X,-\langle \la,\rho \rangle}  \otimes \ol{\D}_{X,-\langle
    \la,\rho \rangle})
  \cdot \Psi_\la(\chi) \simeq \V^K_{\chi,\la} \otimes
  \ol{\V}^K_{\chi,-w_0(\la)}.
  $$\qed
\end{remark}

Corollary \ref{monG} is an analogue of Corollary \ref{unique}. We are
going to formulate a conjectural analogue of Theorem \ref{mainthm}
(Conjecture \ref{eqHG}) in a similar way.

Let $\V_{\la}^{\on{univ}}$ be the
universal vector bundle over $\on{Op}^\ga_{\LG}(X) \times X$ with a
partial connection $\nabla^{\on{univ}}$ along $X$, such that
$$
(\V_{\la}^{\on{univ}},\nabla^{\on{univ}})|_{\chi \times X} =
(\V_{\la},\nabla_{\chi,\la}), \qquad \chi \in \on{Op}^\ga_{\LG}(X).
$$
Let $\pi:
\on{Op}^\ga_{\LG}(X) \times X \to X$ be the projection and set
$$
\V_{X,\la}^{\on{univ}} := \pi_*(\V_{\la}^{\on{univ}}).
$$
Then $K_X^{-\langle \la,\rho
  \rangle} \otimes \V_{X,\la}^{\on{univ}}$ is naturally a $\D_{X,-\langle
  \la,\rho \rangle}$-module on $X$, equipped with a commuting action
of $\on{Fun} \on{Op}^\ga_{\LG}(X) \simeq D_G$.

Moreover, the oper Borel reduction gives rise to an embedding
$$
\kappa_{\la}^{\on{univ}}: K_X^{-\langle \la,\rho \rangle} \hookrightarrow
\V_{\la,X}^{\on{univ}}
$$
and hence a canonical section
$$
s_{\la}^{\on{univ}} \in \Gamma(X,K_X^{-\langle \la,\rho
  \rangle} \otimes \V_{X,\la}^{\on{univ}}).
$$
Consider the cyclic $D_G \otimes \D_{X,-\langle \la,\rho
  \rangle}$-module generated by $s_{\la}^{\on{univ}}$:
$$
\V_{X,\la}^{K,\on{univ}} :=  (D_G \otimes \D_{X,-\langle \la,\rho
  \rangle}) \cdot s_{\la}^{\on{univ}}.
$$

Now recall that the Hecke operator $\wh{H}_\la$ is a section of
$\Omega_X^{-\langle \la,\rho \rangle}$ with values in operators $V_G(\la)
\to V_G$. Hence we can apply to it the sheaf $\D_{X,-\langle \la,\rho
  \rangle}$ as well as the algebra $D_G$, through its action on
$V_G$. The two actions commute, and they generate a $\D_{X,-\langle
  \la,\rho \rangle}$-module inside the sheaf of $C^\infty$ sections of
$\Omega_X^{-\langle \la,\rho \rangle}$ on $X$ with values in operators
$V_G(\la) \to V_G$. Let us denote this $\D_{X,-\langle
  \la,\rho \rangle}$-module by $\langle H_\la \rangle$.

Similarly, we can apply to $\wh{H}_\la$ the sheaf $\ol{\D}_{X,-\langle
  \la,\rho \rangle}$ and the algebra $\ol{D}_G$. Denote the resulting
$\ol{\D}_{X,-\langle \la,\rho \rangle}$-module by $\ol{\langle H_\la
  \rangle}$. The following is an analogue Theorem \ref{mainthm} for
a general group $G$.

\begin{conjecture}    \label{eqHG}
  There are isomorphisms
  \begin{equation}
    \langle H_\la
  \rangle \simeq \V_{X,\la}^{K,\on{univ}}, \qquad \ol{ \langle H_\la
      \rangle } \simeq \ol\V_{X,\la}^{K,\on{univ}}
  \end{equation}
  of $D_G \otimes \D_{X,-\langle \la,\rho \rangle}$-modules
  (resp. $\ol{D}_G \otimes \ol{\D}_{X,-\langle \la,\rho
    \rangle}$-modules).
\end{conjecture}

In the case $G=PGL_n, \la=\omega_1$, this conjecture equivalent to the
statement of Theorem \ref{mainthm}.

Suppose that the monodromy representation of the flat vector bundle
$(V_\la,\nabla_{\chi,\la})$, where $\chi \in
\on{Op}^\ga_{\LG}(X)_{\R}$, is irreducible.\footnote{This
irreducibility assumption has been removed in \cite{EFK4}, Proposition
3.39.} Then Conjecture \ref{eigHG} follows from Conjecture \ref{eqHG}
and Corollary \ref{monG} in the same way as Corollary \ref{eigH1}
follows from Theorem \ref{mainthm} and Corollary \ref{unique} in the
case $G=PGL_n, \la=\omega_1$.

\subsection{Hecke eigensheaf property}    \label{HeckeGen}

As in the case of $PGL_n$,
we wish to derive Conjecture \ref{eqHG} using the formalism of Section
\ref{inttr} from the Hecke eigensheaf property established by
Beilinson and Drinfeld \cite{BD}.

We start by recalling the definition of the Hecke functor in the
general case from \cite{BD}.

We will use the notation of Section \ref{notation}. Consider the Hecke
correspondence $\ol{Z}(\la)$. The fibers of the morphism $\ol{q}_2
\times \ol{q}_3$ are isomorphic to the closure $\ol\Gr_\la$ of the
$G[[z]]$-orbit $\Gr_\la$ in the affine Grassmannian of $G$. We have
denoted by $Z(\la)$ the open dense part of $\ol{Z}(\la)$ such that the
fibers of $q_2 \times q_3$ restricted to $Z(\la)$ are isomorphic to
$\Gr_\la$, and we have denoted by $q_i$ the restriction of the
morphism $\ol{q}_i$ to $Z(\la)$.

The morphism $\ol{q}_2 \times \ol{q}_3$ is
  smooth (in the sense of algebraic geometry) if and only if $\la$ is
  minuscule. (In this case, we also have $\ol\Gr_\la = \Gr_\la$.)

The following definition is taken from \cite{BD}, Sect. 5.2.4.

\begin{definition}[\cite{BD}]
  {\em Let
  $M$ be a left $D$-module on $\Bun_G$. Denote by $q_1^\star(M)$ the
  intermediate extension to $Z(\la)$ of $q_1^*(M)$. (Locally, we
  can choose an isomorphism $\ol{Z}(\la) \simeq \Bun_G \times
  \ol\Gr_\la \times X$ so that $\ol{q}_1$ is the projection on the
  first factor; then $q_1^\star(M)$ can be identified with the
  exterior tensor product of $q_1^*(M)$, the irreducible $D$-module on
  $\ol\Gr_\la$, and $\OO_X$.) The Hecke functor is defined by the
  following formula:}
\begin{equation}    \label{Tla1}
{\mb H}^\dt_\la(M) := (\ol{q}_2 \times \ol{q}_3)^D_*(q_1^\star(M)).
\end{equation}
\end{definition}

Now consider the left $D$-module $\D_{\Bun_G} \otimes \mcL^{-1}$,
where $\mcL = K_{\Bun}^{1/2}$. It is equipped with the commuting right
action of
$$
D_G = \Gamma(\Bun^\beta_G, \mcL \otimes \D_{\Bun_G} \otimes \mcL^{-1})
\simeq \on{Fun} \on{Op}^\ga_{\LG}(X)
$$
(see \cite{BD}, Sect. 5.1.1).

The isomorphism \eqref{cana} gives rise to a section
$\psi_{Z(\la)}$ of $$(\mcL \boxtimes K_X^{-\langle \la,\rho
  \rangle}) \otimes {\mb H}^0_\la(\D_{\Bun_G} \otimes \mcL^{-1}).$$

The following theorem is due to \cite{BD} (see the references before
Theorem \ref{Heceigen} above).

\begin{theorem}[\cite{BD}]    \label{Heceigen1}
  ${\mb H}^i_{\la}(\D_{\Bun_G}
  \otimes \mcL^{-1}) = 0$ for $i \neq 0$, and
  $${\mb H}^0_{\la}(\D_{\Bun_G}
  \otimes \mcL^{-1}) \simeq (\D_{\Bun_G}
  \otimes \mcL^{-1}) \underset{D_G}\boxtimes
  \V_{\la,X}^{\on{univ}}
  $$
  as left $\D_{\Bun_G} \boxtimes \D_X$-modules equipped with a
  commuting action of $D_G$.
  
  Moreover, the section
  $\psi_{Z(\la)}$ of
  $$(\mcL \boxtimes K_X^{-\langle \la,\rho \rangle}) \otimes {\mb
      H}_\la(\D_{\Bun_G} \otimes \mcL^{-1}) \simeq (\mcL \otimes
    \D_{\Bun_G} \otimes \mcL^{-1}) \underset{D_G}\boxtimes
    (K_X^{-\langle \la,\rho \rangle} \otimes \V_{\la,X}^{\on{univ}})
  $$
  coincides with $1 \boxtimes s_{\la}^{\on{univ}}$.
\end{theorem}

Theorem \ref{Heceigen1} implies

\begin{corollary}    \label{psiZla}
  There is an isomorphism
  \begin{equation}
  (D_G \otimes \D_{X,-\langle \la,\rho \rangle}) \cdot
  \psi_{Z(\la)} \simeq \V_{X,\la}^{K,\on{univ}}
  \end{equation}
  of $D_G \otimes \D_{X,-\langle \la,\rho \rangle}$-modules.
\end{corollary}

Conjecture \ref{eqHG} (and hence Conjecture \ref{eigHG} if the
monodromy representation of the flat vector bundle
$(V_\la,\nabla_{\chi,\la})$ is irreducible) can be
derived from Corollary \ref{psiZla} by adapting the results of Section
\ref{inttr} to the present situation (similarly to what we did in the
proof of Theorem \ref{mainthm} in Section \ref{derive} in the case of
$PGL_n$). For non-minuscule $\la$, this requires
additional care since the morphism $\ol{q}_2 \times \ol{q}_3$ is not
smooth in this case (in the sense of algebraic geometry). We leave the
details to a follow-up paper.

\subsection{Proof of Theorem \ref{h}}    \label{proofh}

In this subsection we derive Theorem \ref{h} from a local statement
about line bundles on the affine Grassmannian given in formula (241)
of \cite{BD} (it is reproduced in formula \eqref{241} below). All
ingredients are contained in \cite{BD}. We include the argument here
for completeness.

For a point $x$ of our curve $X$, let $F_x$ be the formal completion
of the field of rational functions on $X$ at $x$ and $\OO_x$ its ring
of integers. We will also use the notation $O = \C[[z]], F =
\C\zpart$. Let $\Aut O$ be the group of automorphisms
of $O$. It naturally acts on the formal loop group $G(F)$ preserving
the subgroup $G(O)$. Hence we obtain an action of $\Aut O$ on the
affine Grassmannian $\Gr = G(F)/G(O)$, which preserves the
$G(O)$-orbits $\Gr_\la, \la \in P^\vee_+$.

Consider first the case when the set $\{ \langle \la,\rho \rangle, \la
\in P^\vee_+ \}$ only contains integers. Let $\wh{M}$ be the
ind-scheme defined in \cite{BD}, Sect. 2.8.3, which parametrizes
quadruples $(x,t_x,\mcF,\ga_x)$, where $x$ is a point of our curve
$X$, $t_x$ is a formal coordinate at $x$ (so that we can identify
$\OO_x$ with $C[[t_x]]$), $\mcF$ is a $G$-bundle on $X$, and $\ga_x$
is a trivialization of $\mcF$ on the disc $D_x = \on{Spec} \OO_x$. The
projection
\begin{eqnarray}    \label{proj}
  \wh{M} &\to& \Bun_G \times X \\ \notag
  (x,t_x,\mcF,\ga_x) &\mapsto& (\mcF,x)
\end{eqnarray}
is a torsor for the group $\Aut O \ltimes G(O)$, which naturally acts
on $t_x$ and $\ga_x$. We use it to construct a functor ${\mathbf F}$ sending a
(ind-)scheme $Y$ with an action of $\Aut O \ltimes G(O)$ to a
(ind-)scheme over $\Bun_G \times X$,
$$
{\mathbf F}(Y) = {\mc Y} := \wh{M} \underset{\Aut O \ltimes G(O)}\times Y
$$
This is a generalization of the Gelfand--Kazhdan functor \cite{GK}
from (ind-)schemes equipped with an action of the group $\Aut O$ to
(ind-)schemes over $X$. Applying it to $\Gr_\la$, we obtain the scheme
$$
{\mathbf F}(\Gr_\la) = {\mc Gr}_\la := \wh{M} \underset{\Aut O \ltimes
  G(O)}\times \Gr_\la
$$
over $\Bun_G \times X$. Denote by $r$ the projection ${\mc Gr}_\la \to
\Bun_G \times X$.

\begin{proposition}[\cite{BD}, Sect. 5.2.2(ii)]    \label{identGrH}
  There is a natural isomorphism $Z(\la) \simeq {\mc Gr}_\la$ under
  which the projection $q_2 \times q_3: Z(\la) \to \Bun_G \times X$
  is identified with $r$.
\end{proposition}

Next, in \cite{BD}, Sect. 4.6, a local Pfaffian line bundle was
defined on $\Gr$. We will denote it by ${\mc L}_{\Gr}$. According to
the construction, ${\mc L}_{\Gr}$ is $\Aut O \ltimes
G(O)$-equivariant, and hence so is its restriction ${\mc L}_{\Gr_\la}$
to the $G(O)$-orbit $\Gr_\la$. Applying the above functor ${\mathbf F}$ to it,
we obtain a line bundle on ${\mc Gr}_\la$, which we will denote by
${\mc L}_{{\mc Gr}_\la}$.

In \cite{BD}, the line bundle ${\mc L}_{\Gr_\la}$ was described
explicitly. To explain this result, we need to define a certain
one-dimensional representation of $\Aut O \ltimes G(O)$. Namely, let
us assign to every element $\phi$ of $\Aut O$ the the image of $z \in
\C[[z]] = O$ under $\phi$. This assignment sets up a bijection between
$\Aut O$ and the space of formal power series
\begin{equation}    \label{rho}
\phi(z) = \sum_{n\geq 0} \phi_n z^{n+1} \in \C[[z]],
\end{equation}
where $\phi_0$ is invertible. In particular, we obtain a canonical
homomorphism $\ga: \Aut O \to {\mathbb G}_m, \phi(z) \mapsto \phi_0$.

\begin{definition}
{\em For $n \in \Z$, let ${\mathfrak o}(n)$ be the one-dimensional
representation of $\Aut O$ on which it acts via the composition of
$\ga$ and the character of ${\mathbb G}_m$ raising $\phi_0$ to the
power $n$. We extend it to a one-dimensional representation (denoted
in the same way) of $\Aut O \ltimes G(O)$.}
\end{definition}

\begin{theorem}[\cite{BD}, formula (241)]    \label{th241}
There is a canonical isomorphism of $\Aut O \ltimes G(O)$-equivariant
line bundles on $\Gr_\la$,
\begin{equation}    \label{241}
  {\mc L}_{\Gr_\la} \simeq K_{\Gr_\la} \otimes {\mathfrak o}_\la,
\end{equation}
where $K_{\Gr_\la}$ is the canonical line bundle on $\Gr_\la$ and
${\mathfrak o}_\la := {\mathfrak o}(-\langle \la,\rho \rangle)$.
\end{theorem}

We now derive Theorem \ref{h} from this result. First, we need the
following statement which is proved e.g. in \cite{FB}, Sect. 6.4.

\begin{lemma}    \label{o}
Under the functor ${\mathbf F}$ introduced above, the representation ${\mathfrak
  o}(n)$ of $\Aut O \ltimes G(O)$ goes to the line bundle
$r_X^*(K_X^n)$, where $r_X$ is the projection ${\mc Gr}_\la
\overset{r}\to \Bun_G \times X \to X$.
\end{lemma}

The functor ${\mathbf F}$ also sends $K_{\Gr_\la}$ to the
relative canonical line bundle of the morphism $r: {\mc Gr}_\la \to
\Bun_G \times X$, which by Proposition \ref{identGrH} is the line
bundle $K_2$ introduced in Section \ref{notation}. We also have $r_X
= q_3$ under the isomorphism of Proposition \ref{identGrH}. Therefore,
Theorem \ref{th241} and Lemma \ref{o} imply that there is a canonical
isomorphism
\begin{equation}    \label{241gl}
  {\mc L}_{{\mc Gr}_\la} \simeq K_2 \otimes
  q_3^*(K_X^{-\langle \la,\rho \rangle}).
\end{equation}

On the other hand, under our current assumption that the set $\{
\langle \la,\rho \rangle, \la \in P^\vee_+ \}$ only contains integers,
Beilinson--Drinfeld construction in \cite{BD}, Sect. 4.4.1, produces a
square root of the canonical line bundle on $\Bun_G$ (normalized by a
trivialization of its fiber at the trivial $G$-bundle). We will denote
it by $K^{1/2}_{\Bun}$. The results of \cite{BD}, Sects. 4.4.14 and
4.6, imply that under the isomorphism of Proposition \ref{identGrH},
we have a canonical identification
\begin{equation}    \label{canident}
{\mc L}_{{\mc Gr}_\la} \simeq q_1^*({K^{1/2}_{\Bun}}) \otimes
q_2^*(K^{1/2}_{\Bun})^{-1}.
\end{equation}
Combining the isomorphisms \eqref{241gl} and \eqref{canident}, we
obtain the isomorphism \eqref{cana}. This completes the proof of
Theorem \ref{h} under this assumption.

Now suppose that the set $\{ \langle \la,\rho \rangle, \la \in
P^\vee_+ \}$ contains half-integers. Then we modify the above argument
as follows. Define a double cover $\on{Aut}_2 O$ of the group $\Aut O$
as the subgroup of the group $\on{Aut} O \times {\mathbb G}_m$
consisting of pairs $(\phi,w)$, where $\phi$ is given by formula
\eqref{rho} and $w^2 = \phi_0$. Define the homomorphism $\gamma_2:
\on{Aut} O \to {\mathbb G}_m$ by the formula
\begin{equation}    \label{ga2}
\gamma_2(\phi,w)=w.
\end{equation}

Next, we fix a square root $K_X^{1/2}$ of the canonical line bundle
$K_X$ on $X$. As explained in \cite{BD}, Sect. 4.3.16, we can then
extend the above functor ${\mathbf F}$ to a functor ${\mathbf F}_2$
from (ind-)schemes $Y$ with an action of $\Aut_2 O \ltimes G(O)$ to
(ind-)schemes over $\Bun_G \times X$, which has the following defining
property. Let ${\mathfrak o}(m), m \in \frac{1}{2} \Z$, be the
one-dimensional representation of $\Aut_2 O \ltimes G(O)$ on which
$G(O)$ acts trivially and $\Aut_2 O$ acts as the composition of
$\ga_2$ given by formula \eqref{ga2} and the one-dimensional
representation of ${\mathbb G}_m$ given by $w \mapsto w^{2m}$. Then
${\mathbf F}_2$ sends ${\mathfrak o}(m)$ to the line bundle
$r_X^*(K_X^m)$ for all $m \in \frac{1}{2} \Z$ (here $K_X^m$ stands for
$(K_X^{1/2})^{2m}$, where $K_X^{1/2}$ is the chosen square root of
$K_X$).

As shown in \cite{BD}, a local Pfaffian line bundle ${\mc L}_{\Gr}$ on
$\Gr$ can still be defined in this case, but it is now $\Aut_2 O
\ltimes G(O)$-equivariant. Moreover, the isomorphism \eqref{241} then
holds as an isomorphism of $\Aut_2 O \ltimes G(O)$-equivariant line
bundles on $\Gr_\la$ and we also have the isomorphism
\eqref{canident}, where $K^{1/2}_{\Bun}$ denotes the square root of
the canonical line bundle on $\Bun_G$ associated to the above choice
of $K^{1/2}_X$ (see the discussion before Theorem \ref{h}). Therefore
the same argument proves the isomorphism \eqref{cana} in general. This
completes the proof of Theorem \ref{h}.


\begin{thebibliography}{GKM}

\bibitem[AGKRRV]{AGKRRV} D. Arinkin, D. Gaitsgory, D. Kazhdan,
S. Raskin, N. Rozenblyum, Y. Varshavsky, {\em The stack of local
  systems with restricted variation and geometric Langlands theory
  with nilpotent singular support}, arXiv:2010.01906.

\bibitem[BB]{BB} A. Beilinson and J. Bernstein, {\em A proof of
    Jantzen conjectures}, in I.M.Gelfand Seminar,
  eds. S. Gelfand, S. Gindikin, Adv. in Soviet Math. {\bf 16}, Part 1,
  pp. 1--50, Providence, AMS, 1993.

\bibitem[BD1]{BD} A.~Beilinson and V.~Drinfeld, {\em Quantization of
    Hitchin's integrable system and Hecke eigensheaves}, Preprint
  available at \newline
  \url{http://www.math.uchicago.edu/~drinfeld/langlands/QuantizationHitchin.pdf}

\bibitem[BD2]{BD:opers} A.~Beilinson and V.~Drinfeld, {\em Opers},
  arXiv:math/0501398.

\bibitem[BK]{BK} A. Braverman and D. Kazhdan, {\em Some examples of
    Hecke algebras for two-dimensional local fields}, Nagoya
  J. Math. {\bf 184} (2006) 57--84.

\bibitem[DS]{DS} A. D'Agnolo and P. Schapira, {\em Radon--Penrose
  Transform for ${\mc D}$-Modules}, J. Funct. Analysis {\bf 139} (1996)
  349--382.

\bibitem[DrS]{DrS} V.G. Drinfeld and V.V. Sokolov, {\em Equations of
  Korteweg-deVries type and simple Lie algebras}, Soviet Math.
  Dokl. {\bf 23} (1981) 457--462.

\bibitem[EFK1]{EFK} P. Etingof, E. Frenkel, and D. Kazhdan, {\em An
  analytic version of the Langlands correspondence for complex
  curves}, in Integrability,
Quantization, and Geometry, dedicated to Boris Dubrovin, Vol. II,
eds. S. Novikov, e.a., pp. 137--202, Proc. Symp. Pure Math. {\bf
  103.2}, AMS, 2021 (arXiv:1908.09677).

\bibitem[EFK2]{EFKnew} P. Etingof, E. Frenkel, and D. Kazhdan, {\em
  Analytic Langlands correspondence for $PGL_2$ on $\Bbb P^1$ with
  parabolic structures over local fields}, Geometric and Functional
  Analysis {\bf 32} (2022) 725--831 (arXiv:2106.05243).

\bibitem[EFK3]{EFK4} P. Etingof, E. Frenkel, and D. Kazhdan, {\em
  A general framework for the analytic Langlands correspondence}, to
  appear in Pure Appl. Math. Quart. (arXiv:arXiv:2311.03743).

\bibitem[FR]{FR} J. F\ae{}rgeman and S. Raskin,
{\em Non-vanishing of geometric Whittaker coefficients for reductive
  groups}, arXiv:2207.02955.

\bibitem[Fa]{F} G. Faltings, {\em Real projective structures on
    Riemann surfaces}, Compositio Mathematica, {\bf 48} (1983)
223--269.

\bibitem[FF]{FF} B.~Feigin and E.~Frenkel, {\em Affine Kac-Moody
    algebras at the critical level and Gelfand-Dikii algebras},
  Int. J. Mod. Phys. {\bf A7}, Suppl. 1A (1992) 197--215.

\bibitem[Fr1]{F:wak} E.~Frenkel, {\em Wakimoto modules, opers and the
    center at the critical level}, Adv. Math. {\bf 195} (2005)
  297--404 (arXiv:math/0210029).

\bibitem[Fr2]{F:analyt} E. Frenkel, {\em Is there an analytic theory of
    automorphic functions for complex algebraic curves?}  SIGMA {\bf
  16} (2020) 042 (arXiv:1812.08160).

\bibitem[FB]{FB} E.~Frenkel and D.~Ben-Zvi, {\em Vertex Algebras and
  Algebraic Curves}, Mathematical Surveys and Monographs {\bf 88},
  Second Edition, AMS, 2004.

\bibitem[FT]{FT} E. Frenkel and C. Teleman, {\em Geometric Langlands
  Correspondence Near Opers}, J. Ramanujan Math. Soc. {\bf 28A} (2013)
  123--147 (arXiv:1306.0876).

\bibitem[Ga]{Gai:outline} D. Gaitsgory, {\em Outline of the proof of
    the geometric Langlands conjecture for $GL_2$}, Ast\'erisque {\bf
    370} (2015) 1--112.

\bibitem[GKM]{GKM} D. Gallo, M. Kapovich, and A. Marden, {\em The
    monodromy groups of Schwarzian equations on closed Riemann
    surfaces}, Ann. Math. {\bf 151} (2000) 625--704.

\bibitem[GK]{GK} I.M. Gelfand and D. Kazhdan, {\em Some problems of
    differential geometry and the calculation of cohomologies of Lie
    algebras of vector fields}, Soviet Math. Dokl. {\bf 12} (1971)
  1367--1370.

\bibitem[Gol]{Go} W. Goldman, {\em Projective structures with Fuchsian
  holonomy}, J. Diff. Geom. {\bf 25} (1987) 297--326.

\bibitem[Gon]{Gon} A. B. Goncharov, {\em Differential Equations and
  Integral Geometry}, Adv. Math. {\bf 131} (1997) 279--343.

\bibitem[Ka]{Ka} M. Kashiwara, {\em D-modules and
Microlocal Calculus}, Transactions of Mathematical Monographs {\bf
  217}, AMS, 2003.

\bibitem[Ko]{Ko} M. Kontsevich, {\em Notes on motives in finite
  characteristic}, in Algebra, Arithmetic, and Geometry, in honor of
  Yu. I. Manin, vol. II, Progress in Mathematics {\bf 270}, Springer,
  2010, pp. 213--247 (arXiv:math/0702206).

\bibitem[L]{L:analyt} R.P.~Langlands, {\em On the analytic form of
    the geometric theory of automorphic forms}, Preprint
available at \url{http://publications.ias.edu/rpl/section/2659}

\bibitem[LS]{LS} Y. Laszlo and Ch. Sorger, {\em The line bundles on the
  moduli of parabolic $G$-bundles over curves and their sections},
  Ann. Sci. \'Ecole Norm. Sup. (4) {\bf 30} (1997) 499--525.

\bibitem[La]{L} G. Laumon, {\em Transformation de Fourier
  g\'{e}n\'{e}ralis\'{e}e}, Preprint alg-geom/9603004.

\bibitem[NR]{NR} M.S. Narasimhan and S. Ramanan, {\em Geometry of
  Hecke cycles}, in C. P. Ramanujam -- A Tribute, Tata Institute of
  Fundamental Research Studies in Mathematics {\bf 8}, pp. 291--345,
  Springer, 1978.

\bibitem[R]{R} M. Rothstein, {\em Connections on the total Picard sheaf
and the KP hierarchy}, Acta Applicandae Mathematicae {\bf 42} (1996)
297--308.

\bibitem[T]{Teschner} J. Teschner, {\em Quantisation conditions of
    the quantum Hitchin system and the real geometric Langlands
    correspondence}, in Geometry and Physics, Festschrift in honour of
  Nigel Hitchin, Vol. I, eds. A. Dancer, J.E. Andersen, and
  O. Garcia-Prada, pp. 347--375, Oxford University Press, 2018
  (arXiv:1707.07873).

\bibitem[W]{We} A. Weil, {\em Ad\`eles et groupes alg\'ebriques},
  S\'eminaire Bourbaki, {\bf 5}, Exp. 186, pp. 249--257, 1959.

\end{thebibliography}
\end{document}